\documentclass[EJP]{ejpecp} 

 \numberwithin{equation}{section}
\theoremstyle{plain}

\usepackage{enumerate}
\usepackage{subfigure}
\usepackage{xcolor}
\usepackage{geometry}
\usepackage{marginnote}

 \SHORTTITLE{\mytitle}

 \TITLE{\mytitle} 

 \AUTHORS{%
Mir-Omid Haji-Mirsadeghi\footnote{Sharif University of Technology, Tehran, Iran.
\EMAIL{mirsadeshi@sharif.ir}}
\and 
Ali Khezeli\footnote{Sharif University of Technology, Tehran, Iran. \EMAIL{alikhezeli@mehr.sharif.ir} The work of this author is part of his PhD program under supervision of Kasra Alishahi and Mir-Omid Haji-Mirsadeghi.}}


 \AMSSUBJ{60G57; 60G55; 60G10} 

 \SUBMITTED{{April 14, 2015}} 
\ACCEPTED{{}} 

 \ARXIVID{1504.02965v2} 

 \VOLUME{0}
\YEAR{2015}
\PAPERNUM{0}
\DOI{vVOL-PID}

 \ABSTRACT{
\verified{
	We give an algorithm to construct a translation-invariant transport kernel between ergodic \verified{stationary} random measures $\Phi$ and $\Psi$ on $\mathbb R^d$, {given that they have equal intensities}.
	As a result, \verified{this} yields a construction of a shift-coupling of an ergodic \verified{stationary} random measure and its Palm version.
	This algorithm constructs the transport kernel in a deterministic manner given realizations $\varphi$ and $\psi$ of the measures.
	The (non-constructive) existence of such a transport kernel was proved in~\cite{ThLa09}. Our algorithm is a generalization of the work of~\cite{HoPe06}, in which  a construction is provided for the Lebesgue measure and an ergodic simple point process.
	In the general case, we limit ourselves to what we call \textit{constrained densities and transport kernels}. 
	We give a definition of stability of \verified{constrained} densities and introduce {our} construction algorithm inspired by the Gale-Shapley stable marriage algorithm.
	For stable \verified{constrained} densities, we study existence, uniqueness, monotonicity w.r.t. {the measures} and boundedness.
}
}

\numberwithin{equation}{section}
\theoremstyle{plain}

\newcommand{\new}[1]{{#1}}

\newcommand{\better}[1]{{#1}}

\newcommand{\verified}[1]{{#1}}
\newcommand{\mar}[1]{}

\theoremstyle{theorem}
	\newtheorem{algorithm}[theorem]{Algorithm}
	\newtheorem{openproblem}[theorem]{Open Problem}
\theoremstyle{definition}
\theoremstyle{remark}
	
\numberwithin{equation}{section}

 \newcommand{\defstyle}[1]{\textit{\textbf{#1}}}
\newcommand{\abs}[1]{\left|#1\right|}
\newcommand{\dist}[2]{\abs{#1-#2}}
\newcommand{\identity}{\mathbb{1}}
\newcommand{\interior}[1]{#1^{\circ}}
\newcommand{\ball}[2]{B_{#2}\left(#1\right)}
\newcommand{\ballint}[2]{B^{\circ}_{#2}\left(#1\right)}

 \newcommand{\myprob}[1]{\mathbb P \left[ #1 \right]}
\newcommand{\omid}[1]{\mathbb E \left[ #1 \right]}
\newcommand{\omidPalm}[2]{\mathbb E_{#1} \left[ #2 \right]}

\newcommand{\norm}[1]{\left| #1 \right|}

\newcommand{\conv}[1]{Conv\left( #1 \right)}

\newcommand{\g}[4]{{#4}_{#1}(#2,#3)}
\newcommand{\h}[4]{{#4}^{#1}(#2,#3)}
\newcommand{\vtk}{Voronoi transport kernel}
\newcommand{\vdensity}{Voronoi density}
\newcommand{\vterritory}{Voronoi territory}
\newcommand{\vterritories}{Voronoi territories}

 \newcommand{\myXi}{\mathbb R^d}
\newcommand{\myX}{\mathbb R^d}

 \newcommand{\myemph}[1]{\emph{#1}}

 \newcommand{\continuity}[1]{}
\newcommand{\mytitle}{Stable transports between stationary random measures}

\begin{document}

\section{Introduction}

For a random measure $\Psi$ on $\mathbb R^d$, there are a number of equivalent definitions for \myemph{the Palm distribution} of $\Psi$ \verified{(see Section~\ref{sec:preliminaries})}.
Heuristically speaking, in the case that $\Psi$ is {stationary and ergodic}, the \myemph{Palm version} of $\Psi$ is obtained by viewing $\Psi$ from a \myemph{typical point} of $\Psi$.
\verified{The result of~\cite{Th96}} shows that there exists a random \verified{point} $Y$ such that
\verified{by viewing $\Psi$ from $Y$; i.e.}
by translating $\Psi$ by vector $-Y$, we get
 exactly the Palm \verified{version} of $\Psi$. \verified{In other words,} there exists a coupling of $\Psi$ and its Palm version such that almost surely each one is a translated version of the other one. Such a coupling is called a \myemph{shift-coupling}. 
 

To obtain a shift-coupling, one can use a {(random)} \myemph{balancing transport kernel} $T$ that transports \verified{a multiple of} 
 the Lebesgue measure to {$\Psi$}, where by a {(random)} transport kernel we mean a {(random) function} that assigns to {each} point $s \in \mathbb R^d$ {and $\omega$ in the \verified{probability} space,} a \verified{probability} measure $T_{\omega}(s,\cdot)$ on $\mathbb R^d$. This measure can be interpreted as how the infinitesimal mass at $s$ is distributed in the space. Given that $T$ depends on $\Psi$ in a translation-invariant manner \verified{(which is called \textit{flow-adapted} here)}, then choosing $Y$ with distribution $T_{\omega}(0,\cdot)$ gives {a shift-coupling} \verified{of $\Psi$ and its Palm version}
\verified{(see~\cite{ThLa09} and~\cite{HoPe05} {as explained after Theorem~\ref{thm:extraHead}}).}
\verified{ A case of special interest is when $T$ is a {balancing}  \textit{allocation} in which the measure $T_{\omega}(s,\cdot)$ is a Dirac measure {a.e.} In this case, \verified{given $\Psi$}, the vector $Y$ defined above is deterministic and the shift-coupling does not need extra randomness.}

\verified{
Let $\Phi$ and $\Psi$ be jointly stationary and ergodic random measures on $\mathbb R^d$. Considering flow-adapted transport kernels balancing $\Phi$ and $\Psi$ has been of great interest recently. The abstract existence of such transports was proved in~\verified{\cite{ThLa09} (which is merely based on~\cite{Th96})} provided that the intensities are equal, finite and positive. Nevertheless, there has been tremendous interest in the construction of such transports in special cases in the recent years. One of the reasons for this interest is that such transport kernels lead to explicit shift-coupling of the Palm distributions of $\Phi$ and $\Psi$, as explained above.
The constructions were  motivated by Liggett~\cite{Li}, who constructed a balancing allocation for an ergodic simple point process in dimension one. The landmark in this topic is~\cite{HoPe06} {which generalizes~\cite{Li} to arbitrary dimensions}.
}
\verified{
There are several other constructions in the literature for the case of simple point processes, such as gravitational allocation~\cite{ChPe10}, optimal transport~\cite{HuSt13}, one-sided stable allocation on the line~\cite{LaMoTh14}, etc.
}

%

In this paper we give an algorithm that works \verified{in the general case 
and enables us to
\begin{itemize}
\item
construct a flow-adapted transport kernel balancing two arbitrary jointly stationary and ergodic random measures on $\mathbb R^d$ with equal intensities (Theorem~\ref{thm:site-optimalBalancing})
\item
and constructs a shift-coupling for an arbitrary stationary ergodic random measure on $\mathbb R^d$ and its Palm version (Theorem~\ref{thm:extraHead}).
\end{itemize}
}

\verified{To do this, we generalize the notion of stable allocations introduced in~\cite{HoPe06} to what we call \myemph{stable \verified{constrained} densities}, where the notion of \verified{constrained} densities is a special case of \myemph{capacity constrained transport kernels} introduced in~\cite{KoMC12}. The first, and in fact deterministic, result (Theorem~\ref{thm:stable}) is that stable \verified{constrained} densities exist and one can be given by {our algorithm (Algorithm~\ref{alg:Gale})} which is inspired by the continuum version of the Gale-Shapley algorithm in~\cite{HoPe06}. Another important result is considering the algorithm in the random case \verified{described above}. 
Other results deal with monotonicity (Theorem~\ref{thm:monotonicity}) and optimality (Corollary~\ref{cor:optimality}) properties of stable \verified{constrained} densities, uniqueness (Theorem~\ref{thm:uniqueness}) and boundedness of the support of the mass transported from and to a region (Theorem~\ref{thm:bounded}). These results are in the spirit of the seminal papers~\cite{HoPe06} {and~\cite{HoPe05}} \verified{and generalize some of their results}.
 We also introduce the notion of \myemph{\vtk} with respect to a measure, which generalizes the notion of Voronoi diagram for a discrete set. It {helps} us in proving some statements, but it can be interesting in its own.}
\verified{
The construction and results in this paper can be generalized to random measures on a locally compact Abelian group and also to non-ergodic cases (only equality of \textit{sample intensities} is important) in the setting of~\cite{ThLa09}. We don't go through these general cases to stay focused on the main ideas. 
}

\verified{
We show that using transport kernels is inevitable in the general case by providing examples where no flow-adapted balancing allocation exists. However, some results and open problems on the existence and construction of flow-adapted balancing allocations are addressed in Subsection~\ref{sec:allocations}.
}


\verified{
The paper is structured as follows. Preliminaries about random measures and transport kernels are reviewed in Section~\ref{sec:preliminaries}. Since we generalize some works of~\cite{HoPe06}, in Section~\ref{sec:motivation} we review {the idea of the algorithm in~\cite{HoPe06}} as a motivation of our work. Our main definitions and results are presented in Section~\ref{sec:main} and the proofs are postponed to Section~\ref{sec:proofs}. In Subsection~\ref{sec:mild} we define \verified{constrained} {transports}. We introduce {our algorithm} in Subsection~\ref{sec:shiftcoupling} and state the shift-coupling and balancing properties in the random case. In Subsection~\ref{sec:stability} we generalize the notion of stability which is a key tool for proving the results. Other properties of stable transports are provided in Subsection~\ref{sec:other}. Voronoi transport kernels are defined in Subsection~\ref{sec:voronoi}. In Subsection~\ref{sec:allocations} we study the existence and constructions of balancing allocations. Finally, some examples are given in Section~\ref{sec:examples} which are addressed in the text.
}
\section{Preliminaries}
\label{sec:preliminaries}
Let $\mathcal G$ be the Borel $\sigma$-field on $\mathbb R^d$ and $\mathcal L_d$ be the Lebesgue measure on $(\mathbb R^d,\mathcal G)$. We denote by $M$ the set of all non-negative locally finite measures on $(\mathbb R^d,\mathcal G)$, and by $\mathcal M$ the smallest $\sigma$-field on $M$ such that the mappings $\mu\mapsto\mu(B)$ are measurable for all $B\in
\mathcal G$. All measures in this paper are assumed to be members of $M$.

 {
Recall that the complement, the interior and the boundary of a set $A$ are denoted by $A^c$, $A^{\circ}$ and $\partial A$ respectively.
We denote The identity function on $A$ by $\identity_A$.
Also, the \myemph{closed} ball with center $x\in\mathbb R^d$ and radius $r$ is denoted by $\ball xr$.
}

 For a measure $\varphi$ and a non-negative measurable function $f:\mathbb R^d\rightarrow \mathbb R$, we denote by $f\varphi$ the measure $B\mapsto \int_B{f(s)\varphi(ds)}$.
In this article, by $\varphi_1\geq \varphi_2$ we mean $\varphi_1(B)\geq\varphi_2(B)$ for all $B\in\mathcal G$.
A \defstyle{\verified{weighted} transport kernel} is a measurable map $T:\mathbb R^d\rightarrow M$. For simplicity of notation, we denote $\left( T(s)\right)(B)$ by $T(s,B)$. Intuitively, we can think of $T(s,B)$ as
\verified{ the portion of the infinitesimal mass at $s$ transported to}
the set $B$. \verified{Given measures $\varphi$ and $\psi$, $T$ is called \defstyle{non-weighted} or \defstyle{Markovian} if 
\[
T(s,\mathbb R^d)=1 \quad \text{for } \varphi \text{-almost all } s\in\mathbb R^d
\]
and it is called 
}
\defstyle{$(\varphi,\psi)$-balancing} if it transports $\varphi$ to $\psi$; i.e.
\begin{equation*}
\int_{\mathbb R^d}T(s,\cdot)\varphi(ds) = \psi(\cdot).
\end{equation*}
\verified{In this paper, by a \defstyle{transport kernel} we mean a non-weighted transport kernel. Being non-weighted is equivalent to the condition that}
the total mass transported from a set $B$, which is $\int_B T(s,\mathbb R^d)\varphi(ds)$, is equal to $\varphi(B)$. 
\verified{Being $(\varphi,\psi)$-balancing} means that the mass transported into $B$ is equal to $\psi(B)$. \verified{If $T$ is $(\varphi,\psi)-balancing$ \verified{transport kernel}, then $T$ is a Markovian kernel which transports $\varphi$ to $\psi$.}

\verified{A \defstyle{$(\varphi,\psi)$-balancing allocation} is a transport kernel $T$ such that for $\varphi$-a.e. $s$, $T(s,\cdot)$ is a Dirac measure $\delta_{\tau(s)}$ and $\varphi(\tau^{-1}(\cdot))=\psi(\cdot)$. See Definition~\ref{def:allocation} for a precise definition of an \textit{allocation}.}
%
%
For more details on transport kernels, see~\cite{ThLa09}.

In this work, we fix a measurable space $(\Omega,\mathcal F)$ equipped with a \defstyle{measurable flow} $\theta_s:\Omega\rightarrow\Omega$ for $s\in\mathbb R^d$; i.e. $(\omega,s)\mapsto \theta_s \omega$ is measurable, $\theta_0$ is the identity on $\Omega$ and
\begin{equation*}
\theta_s \circ \theta_t = \theta_{s+t},\quad \forall s,t\in\mathbb R^d.
\end{equation*}
With an abuse of notation, we use $\theta_s$ also for natural flows on the space of functions on $\mathbb R^d$ or $\mathbb R^d\times\mathbb R^d$, on $M$, and on the space of weighted transport kernels; i.e.
\begin{eqnarray*}
\theta_s f(x)&:=&f(x+s),\\
\theta_s f(x,y)&:=&f(x+s,y+s),\\
\theta_s \varphi(B)&:=&\varphi(B+s),\\
\theta_s T(x,B)&:=&T(x+s,B+s).
\end{eqnarray*}
Using these conventions, a measurable function $F$ from $\Omega$ to a flow-equipped space is called \defstyle{flow-adapted} if $F(\theta_s\omega)=\theta_s F(\omega)$. We denote $F(\omega)$ by $F_{\omega}$ too.

 A \defstyle{random measure} is a pair $(\mathbb P,\Phi)$, where $\mathbb P$ is a probability measure on $\Omega$ and $\Phi:\Omega\rightarrow M$ is a measurable function. 
The \defstyle{distribution} of $(\mathbb P,\Phi)$ is the push-forward measure $\mathbb P(\Phi^{-1}(\cdot))$ on $M$.

 A probability measure $\mathbb P$ on $\Omega$ is \defstyle{stationary} if it is invariant under $\theta_s$ for all $s$ in $\mathbb R^d$.
A \defstyle{stationary random measure} is a random measure $(\mathbb P,\Phi)$ such that $\mathbb P$ is a stationary probability measure and $\Phi$ is flow-adapted. This implies stationarity in the usual sense, which is translation-invariance of its distribution; i.e. $\myprob{\Phi\in A}=\myprob{\theta_s \Phi \in A}$ for any $A\in\mathcal M$ and $s\in\mathbb R^d$. \verified{Moreover, if $(\mathbb P, \Psi)$ is another stationary random measure (with the same $\mathbb P$), then $\Phi$ and $\Psi$ are \defstyle{jointly stationary} in the usual sense; i.e. the distribution of the pair $(\Phi,\Psi)$ is invariant by translations.} 
When there is no ambiguity about $\mathbb P$, we may say that $\Phi$ is a stationary random measure in the sense given above. \verified{Therefore, when we mention two stationary random measures without definig the probability measure(s), they are assumed to be jointly stationary.}

 A stationary probability measure $\mathbb P$ on $\Omega$ is \defstyle{ergodic}
if for any event $A\in\mathcal F$ that is invariant under all $\theta_s$, we have $\mathbb P(A)\in\{0,1\}$. The stationary random measure $(\mathbb P,\Phi)$ is ergodic if $\mathbb P$ is ergodic. It is clear that if $\Psi$ is a measure derived from $\Phi$ in a translation-invariant manner, then stationarity (resp. ergodicity) of $(\mathbb P,\Phi)$ implies stationarity (resp. ergodicity) of $(\mathbb P,\Psi)$. \verified{For simplicity, we use the term \textit{ergodic random measure} instead of \textit{ergodic stationary random measure}.}

Since we deal with random measures in this work, we have two different notions of `almost everywhere'; one corresponding to probability and the other corresponding to the measures on $\mathbb R^d$. To avoid confusion, we preserve the phrase \myemph{almost sure(ly)} (or \myemph{a.s.}) for probability measures on $\Omega$ and we use \myemph{almost everywhere} and \myemph{almost all} (denoted by \myemph{a.e.} and \myemph{a.a.}) for measures on $\mathbb R^d$.

\begin{lemma}
\label{lemma:infinite}
Let $V\subseteq\mathbb R^d$ be a subset that contains arbitrarily large balls. If $\Phi$ is a stationary random measure such that a.s. $\Phi(\mathbb R^d)>0$, then a.s. $\Phi(V)=\infty$.
\end{lemma}

 The \defstyle{intensity} of a stationary random measure $(\mathbb P,\Phi)$ is the unique constant $\lambda\in[0,\infty]$ such that $\omid{\Phi(B)} = \lambda \mathcal L_d(B)$ for an arbitrary Borel set $B\in\mathcal G$. If $\lambda$ is positive and finite, then the \defstyle{Palm version} of $(\mathbb P,\Phi)$ is the random measure $(\mathbb P_{\Phi},\Phi)$ in which
\begin{equation}
\verified{\label{eq:Palm}}
\mathbb P_{\Phi}(A):=\frac 1{\lambda \mathcal L_d(B)}\omid{\int_{\mathbb R^d} \identity_A(\theta_s)\identity_B(s)\Phi(ds)},
\end{equation}
for all $A\in\mathcal F$, where $\theta_s$ is interpreted as the random element $\theta_s(\omega)$ and $B$ is an arbitrary Borel set $B\in\mathcal G$ with positive and finite Lebesgue measure. \verified{Note that $\Phi$ stays the same and only the probability measure is changed.} It can be shown that stationarity implies that this definition is independent of the choice of $B$ and formalizes the intuition of the Palm distribution \verified{given in the introduction}. The above equation is equivalent to the fact that
\begin{equation}
\label{eq:omidPalm}
\omidPalm{\Phi}{H} = \frac 1{\lambda \mathcal L_d(B)}\omid{\int_{\mathbb R^d} H(\theta_s)\identity_B(s)\Phi(ds)}
\end{equation}
for all non-negative measurable functions $H:\Omega\rightarrow\mathbb R$, where $\mathbb E_{\Phi}$ is expectation with respect to $\mathbb P_{\Phi}$.
{By this equation, }we can think of $\omidPalm{\Phi}{H}$ as averaging $H(\theta_s\omega)$ over the points $s\in B$ and $\omega\in\Omega$.

 The \myemph{refined Campbell theorem } states that for all non-negative measurable functions $H:\Omega\times\mathbb R^d\rightarrow\mathbb R$,

 \begin{equation}
\label{eq:Campbell}
\lambda \omidPalm{\Phi}{\int H(\theta_0,s) ds} = \omid{\int H(\theta_s,s)\Phi(ds)},
\end{equation}
where $ds$ is a short form of $\mathcal L_d(ds)$. If $(\mathbb P,\Psi)$ is also a stationary random measure on the same space, we have Neveu's exchange formula
\begin{equation}
\label{eq:Neveu}
\lambda_{\Phi} \omidPalm{\Phi}{\int H(\theta_0,s)\Psi(ds)}=\lambda_{\Psi} \omidPalm{\Psi}{\int H(\theta_s , -s)\Phi(ds)},
\end{equation}
where $\lambda_{\Phi}$ and $\lambda_{\Psi}$ are the intensities of $\Phi$ and $\Psi$ respectively. Interested readers may refer to~\cite{b: ScWe}, sections~3.3 and~3.4, for more details on the properties of Palm distributions.

 A \defstyle{shift-coupling} of random measures $(\mathbb P_1,\Phi_1)$ and $(\mathbb P_2,\Phi_2)$ is a random vector $Y$ \verified{(possibly on an extension of $\Omega$)}  such that $(\mathbb P_1,\theta_Y(\Phi_1))$ has the same distribution as $(\mathbb P_2,\Phi_2)$. Note that this \verified{provides a coupling of} the distributions of the two random measures such that they are a translated version of each other in every realization. 
 \verified{
 	The necessary and sufficient condition for the existence of a shift-coupling is proved in~\cite{Th96} in a more general setting. However, the proof in~\cite{Th96} is non-constructive (See Theorem~\ref{thm:extraHead}).
 	}
 {Shift-couplings are of special interest in the case that} $(\mathbb P_1,\Phi_1)$ is a stationary random measure and $(\mathbb P_2,\Phi_2)$ is its Palm version. {They can be used to construct the Palm version of a random measure or to reconstruct the random measure from its Palm version \verified{by a random translation} (see Theorem~\ref{thm:extraHead}).} 
{In particular, when $\Phi_1$ is a simple point process,} $Y$ is a called an \defstyle{extra head scheme} for $\Phi_1$ in~\cite{HoPe05}. \verified{The name comes from the fact that if $\Phi_1$ is the point process on $\mathbb Z^d$ formed by the heads in i.i.d. coin tosses (resp. the Poisson point process), then by shifting $Y$ to the origin we get the same thing (in distribution) as placing a head (resp. point) on the origin.} 

\section{\verified{Motivation: Extra head scheme \verified{for a point process}}}
\label{sec:motivation}
\verified{
 
To get an idea of the definitions in this paper, it is helpful to remind the novel algorithm of~\cite{HoPe06} (Algorithm~\ref{alg:GaleSimple}) which we will generalize. This algorithm is inspired by the stable marriage algorithm of Gale and Shapley in bipartite graphs (\cite{GaSh}) and appears to be the first generalization to a continuum setting. The main goal in the paper is to construct a translation-invariant balancing allocation between the Lebesgue measure and a realization of an ergodic stationary simple point process in $\mathbb R^d$.

 Let $\Xi$ be a discrete subset of $\mathbb R^d$.
By \textit{sites} and \textit{centers} we mean the points of $\mathbb R^d$ and $\Xi$ respectively. The idea of the following algorithm is that each site and each center prefer to be allocated as close as possible. Forget the sites that are equidistant from two or more centers.

 \begin{algorithm} 
 \label{alg:GaleSimple}
For each natural number $n$, stage $n$ consists of the following two parts:
\begin{description}
\item[(a)] Each site $x$ \textit{applies} to the closest center to $x$ which has not rejected $x$ at any earlier stage.
\item[(b)] For each center $\xi$, Let $A$ (which depends on $n$ and $\xi$) be the set of sites which applied to $\xi$ in the previous step. Let $B$ be the smallest ball centered at $\xi$ such that $\mathcal L(A\cap B)\geq 1$. Then $\xi$ rejects the sites in $A\backslash B$. 
\end{description}
Easily seen, each site $x$ either is rejected by all centers or for some center $\xi$, $x$ applies to $\xi$ and is never rejected for sufficiently large $n$. In the first case let {$\tau(x)=\infty$} (where $\infty$ is treated as a single point added to the space whose distance to every other point is $\infty$) and in the second case let $\tau(x):=\xi$.
\end{algorithm}

It is proved in~\cite{HoPe06} that the function $\tau$ defined in this algorithm is \textit{stable} \verified{in a sense that is a generalization of {stable matchings,
which is, roughly speaking, }
there are no sites $x_1$ and $x_2$ such that $\dist{x_1}{\tau(x_2)}<\min\{\dist{x_1}{\tau(x_1)},\dist{x_2}{\tau(x_2)}\}$ (see Definition~\ref{def:stableAllocation} for the exact definition).}
 An interesting theorem in~\cite{HoPe06} is that
 if  $\Psi$ is an ergodic simple point process in $\mathbb R^d$ with intensity 1, then almost surely the (random) function $\tau$ defined in Algorithm~\ref{alg:GaleSimple} for the discrete set $\Psi_{\omega}$ defines a balancing allocation between the Lebesgue measure and $\Psi$; i.e.  almost surely $\tau$ is defined for almost every site and $\mathcal L(\tau^{-1}(\xi))=1$ for all centers $\xi$.


One can try to apply Algorithm~\ref{alg:GaleSimple} to find an allocation for two arbitrary measures $\varphi$ and $\psi$ (instead of the Lebesgue measure and a counting measure). For example, let them be the Lebesgue measure restricted on $(-\infty,0]$ and $[0,\infty)$ respectively. The sites and the centers are the points in the supports of $\varphi$ and $\psi$ respectively. In the first step, all sites apply to center $0$, but the next steps are vague. Same issue may happen in the stationary case. 
If we consider the allocation as a transport kernel $ T(x,\cdot):=\delta_{\tau(x)}$, 
our idea is to force the infinitesimal mass at every point to be spreaded instead of being transported to a single point. To do so,
We will force an upper bound on the  measure $T(x,\cdot)$.
Since $x$ wants to apply to the support of $\psi$, a natural upper bound is $T(x,B)\leq\psi(B)$. 
This justifies the choice of the name `\verified{constrained}' (see definitions~\ref{def:mildTransport} and~\ref{def:compatibleDensity}). 
The above example will be addressed in Example~\ref{ex:-+}.

}

\section{\verified{Definitions and Main Results}}
\label{sec:main}
\subsection{\verified{Constrained Transport Kernels and Constrained densities}}
\label{sec:mild}
Let $\varphi$ and $\psi$ be given locally finite non-negative Borel measures on $\mathbb R^d$. 
%
%
\verified{
\begin{definition}
\label{def:mildTransport}
\verified{A weighted transport kernel $T$ is called {\defstyle{$\psi$-constrained}} if}
 $T(x,\cdot)\leq\psi$ for all $x\in\mathbb R^d$.
\end{definition}

\verified{As we will see in Subsection~\ref{sec:allocations}, a $(\varphi,\psi)$-balancing {$\psi$-constrained} transport kernel cannot be an allocation except maybe when $\psi$ is discrete.}


Since $T(x,\cdot)$ is absolutely continuous w.r.t. $\psi$ for any $x$, we will work with its Radon-Nykodim derivative. 
\verified{The following definition and Remark~\ref{rem:kernel} establish the setup in a measurable way.}
}

 \begin{definition}
\label{def:compatibleDensity}
A non-negative measurable function $f(x,\xi)$ on ${\myX}\times\myXi$ is called a \defstyle{\verified{sub-balancing} density} (given $\varphi$ and $\psi$) if
\begin{eqnarray*}
\int_{\myXi}{f(x,\xi)\psi(d\xi)} \leq 1, && \forall x\in {\myX},\\
\int_{{\myX}}{f(x,\xi)\varphi(dx)} \leq 1, && \forall \xi\in\myXi.
\end{eqnarray*}
We call $f$ \defstyle{balancing}  if equality happens for $\varphi$-a.e. $x$ and $\psi$-a.e. $\xi$. Furthermore, $f$ is called \defstyle{\verified{constrained}} if it is \verified{sub-balancing} and $f(x,\xi)\leq 1$ for every $(x,\xi)\in {\myX}\times\myXi$.
\end{definition}

 \begin{remark}
\label{rem:kernel}
A {\verified{constrained}} density $f$ defines a \verified{$\psi$-constrained} weighted transport kernel via
\begin{equation}
\label{eq:kernel}
T(x,B):=\int_{B}f(x,\xi)\psi(d\xi).
\end{equation}
\verified{$f(x,\xi)$ can be thought as the infinitesimal mass going from $x$ to $\xi$.}
If $f$ is balancing, then $T$ is a $(\varphi,\psi)$-balancing transport kernel.  
\verified{Conversely, If $T$ is a $(\varphi,\psi)$-balancing {$\psi$-constrained} transport kernel, then \verified{(a suitable version of)} the Radon-Nykodim derivative of the measure on $\mathbb R^d\times\mathbb R^d$ defined as $(A\times B)\mapsto \int_A T(x,B)\varphi(dx)$, with respect to $\varphi\otimes\psi$, is a balancing \verified{constrained} density. Hence, we might think of $f$ as the \textit{density} of $T$ w.r.t. $\psi$. Similar correspondence holds in the random case.} 
\end{remark}

\verified{Note that the definition of constrained densities depends on both $\varphi$ and $\psi$, but we don't use a prefix for simplicity.}
 The notion of balancing \verified{constrained} densities is a special case of \myemph{capacity constrained} transport kernels defined in~\cite{KoMC12}. 
  It can be interpreted as there is a transportation capacity between any site and any center. 
\verified{ More explicitly, the total mass going from a set $A$ of sites to a set $B$ of centers is at most $\varphi(A)\psi(B)$.}
 
\subsection{\verified{Construction of a Shift-Coupling}}
\label{sec:shiftcoupling}


Let $\varphi,\psi\in M$ be given. 
\verified{
Since we have different procedures regarding $\varphi$ and $\psi$, it is helpful to use the names \defstyle{sites} and \defstyle{centers} for points in the supports of $\varphi$ and $\psi$ respectively, following the terminology of~\cite{HoPe06}.
 Imagine we have two copies of $\mathbb R^d$, in one copy, we have sites and measure $\varphi$ and in the other copy we have centers and measure $\psi$. 
Nevertheless, we will measure the distance between a site and a center as if they are on the same space.
We use roman letters for naming sites and Greek letters for naming centers. 
}

{Here is an overview of Algorithm~\ref{alg:Gale}.} The algorithm consists of infinitely many stages and each stage has two steps. At stage $n$, each site $x_0$ \myemph{applies} to the closest possible centers with weight $A_n(x_0,\cdot):\myXi\rightarrow [0,1]$ (tries to construct $f(x_0,\cdot)$). Then each center $\xi_0$ \myemph{rejects} some of the weights applied to $\xi_0$ if it has reached its capacity. The amount of rejection is denoted by $R_{n}(\cdot,\xi_0)$.
Note that even if $R_{n}(x_0,\xi_0)>0$, $x_0$ will still apply to $\xi_0$ at all later stages.
 The functions $A_n$ and $R_n$ will be non-decreasing with respect to $n$.
\verified{A heuristic for choosing $A_n$ and $R_n$ in the algorithm is that the sites prefer to apply to centers which are \textit{as close as possible}, in a greedy manner. 
Similarly, the centers prefer to 
reject no portion of the applications of the sites which are \textit{as close as possible}.}

For an illustration of Algorithm~\ref{alg:Gale} see \verified{Example~\ref{ex:interval}. Example~\ref{ex:generalization} shows that this algorithm generalizes the algorithm in~\cite{HoPe06}.}
%
 The name \myemph{site-optimal} is justified in Corollary~\ref{cor:optimality}. 

 \begin{algorithm}[site-optimal Gale-Shapley algorithm]
\label{alg:Gale}
{Given measures $\varphi$ and $\psi$ on $\mathbb R^d$}, let the rejection function be zero at the {beginning}; i.e. $R_0(x,\xi):=0$ for all $(x,\xi) \in \mathbb R^d\times\mathbb R^d$. For each natural number $n$, \myemph{stage} $n$ consists of the following steps:
\begin{enumerate}[(i)]
\item For each site $x_0$, define its \myemph{application radius} at stage $n$ as
\begin{equation}
\label{eq:a_n}
a_n(x_0) := \sup\left\{a: \int_{\ball {x_0}a} \left(1-R_{n-1}(x_0,\xi)\right)\psi(d\xi)\leq 1 \right\}.
\end{equation}
\verified{Define the \myemph{$n$-th application function} as}
\begin{equation*}
A_n(x_0,\xi) := \left\{
\begin{array}{ll}
1 & \abs{x_0-\xi}<a_n(x_0),\\
c R_{n-1}(x_0,\xi) + (1-c) & \abs{x_0-\xi}=a_n(x_0),\\
0 & \abs{x_0-\xi}>a_n(x_0),
\end{array} \right.
\end{equation*}
where $c=c_n(x_0)$ is the constant in $[0,1]$ such that we have
\begin{equation}
\label{eq:c_n}
\int_{\myXi} \left(A_n(x_0,\xi)-R_{n-1}(x_0,\xi)\right)\psi(d\xi)= 1 \quad \verified{\text{if } a_n(x_0)<\infty}
\end{equation}
and we let \verified{$c=1$} \verified{(or any arbitrary constant)} if \verified{$a_n(x_0)=\infty$ or} $\psi(\partial \ball{x_0}{a_n(x_0)})=0$.
We say $x_0$ \myemph{applies} to the centers with weight $A_n(x_0,\cdot)$ {at stage $n$}. 
\item
For each center $\xi_0$, define its \myemph{rejection radius} at stage $n$ as
\begin{equation}
\label{eq:r_n}
r_n(\xi_0) := \sup \left\{ r: \int_{\ball{\xi_0}r} A_n(x,\xi_0)\varphi(dx) \leq 1 \right\}.
\end{equation}
Define the \textit{$n$-th rejection function} as
\begin{equation*}
R_{n}(x,\xi_0) := \left\{
\begin{array}{ll}
0 & \abs{x-\xi_0}<r_n(\xi_0),\\
c' A_n(x,\xi_0) & \abs{x-\xi_0}=r_n(\xi_0),\\
A_n(x,\xi_0) & \abs{x-\xi_0}>r_n(\xi_0),
\end{array} \right.
\end{equation*}
{where $c'=c'_n(\xi_0)$ is the constant in $[0,1]$} such that we have
\[
\int_{{\myX}}\left( A_n(x,\xi_0)-R_{n}(x,\xi_0) \right)\varphi(dx)=1 \quad \verified{\text{if } r_n(\xi_0)<\infty}
\]
and we let $c'=0$ if \verified{$r_n(\xi_0)=\infty$ or} $\varphi(\partial \ball{\xi_0}{r_n(\xi_0)})=0$.
We say $\xi_0$ \myemph{rejects} the application weights according to $R_{n}(\cdot,\xi_0)$ { at stage $n$}.
\end{enumerate}
\end{algorithm}

\verified{
It is shown in the proof of Lemma~\ref{lemma:increasing} that $c_n(x_0)$ and $c'_n(\xi_0)$, which are defined in the algorithm, exist and are well defined.
}

 \begin{lemma}
\label{lemma:increasing}
In Algorithm~\ref{alg:Gale}, $A_n,R_n$ and $a_n$ are non-decreasing with respect to $n$, $r_n$ is non-increasing. 
\verified{Moreover, they depend on $\varphi$ and $\omega$ in a measurable and flow-adapted manner.} 
\end{lemma}

 This lemma allows us to provide the following definitions.

 \begin{definition}
\label{def:site-optimalDensity}
In Algorithm~\ref{alg:Gale},
the \myemph{(final) application radius}, \myemph{rejection radius}, \myemph{application function}, \myemph{rejection function} and \defstyle{the site-optimal density} are defined as follows

 \begin{eqnarray*}
a(x)&:=& \lim_{n\rightarrow\infty}a_n(x),\\
r(\xi)&:=& \lim_{n\rightarrow\infty}r_n(\xi),\\
A(x,\xi)&:=& \lim_{n\rightarrow\infty}A_n(x,\xi),\\
R(x,\xi)&:=& \lim_{n\rightarrow\infty}R_n(x, \xi),\\
f_s(x,\xi)&:=& A(x,\xi)-R(x,\xi).
\end{eqnarray*}

 \end{definition}

\verified{
\begin{definition}
The \defstyle{center-optimal} Gale-Shapley algorithm for $\varphi$ and $\psi$ is just Algorithm~\ref{alg:Gale} for $\psi$ and $\varphi$; i.e. we swap the roles of sites and centers. \defstyle{The center-optimal density} $f_c$ is defined similar to Definition~\ref{def:site-optimalDensity} using the center-optimal Gale-Shapley algorithm.
\end{definition}
}
\begin{theorem}
\label{thm:site-optimalBalancing}
\verified{ Let $\Phi$ and $\Psi$ be ergodic 
random measures on $\mathbb R^d$ with positive and finite intensities. If the intensities are equal, then}
the site-optimal density which is constructed by Algorithm~\ref{alg:Gale} is \verified{almost surely a balancing \verified{constrained} density}. 
Therefore, it almost surely gives a flow-adapted  $(\Phi_{\omega}, \Psi_{\omega})$-balancing transport kernel via~\eqref{eq:kernel}.
\end{theorem}


\verified{ {According to Theorem~\ref{thm:site-optimalBalancing}}, {Algorithm~\ref{alg:Gale}} answers a question in~\cite{ThLa09} asking for a constructive algorithm to find a \verified{flow-adapted} balancing transport kernel, although the existence of such a transport kernel is proved to be equivalent to the equality of the \textit{sample intensities} in~\cite{ThLa09}, which is granted here by ergodicity.}
 
 As a result of Theorem~\ref{thm:site-optimalBalancing}, one can construct a shift-coupling between an ergodic random measure and its Palm version as in the following theorem.
 \verified{
 	The key tool for this construction is the following well-known theorem in the literature.
}

 \begin{theorem}[Shift-coupling]
\label{thm:extraHead}
\verified{
Let $\Psi$ be an ergodic random measure on $\mathbb R^d$ with positive and finite intensity $\lambda_{\Psi}$ and let $\Phi=\lambda_{\Psi}\mathcal L_d$.
Let $F=F_{\omega}(x,\xi)$ be a flow-adapted \verified{function which is almost surely a} balancing density for {$\Phi_{\omega}$ and $\Psi_{\omega}$}; e.g.} the site-optimal density for {$\Phi_{\omega}$ and $\Psi_{\omega}$.}
\begin{enumerate}[(i)]
\item
\label{thm:extraHead:1}
If $Y$ is a random vector such that its conditional distribution given $\Psi$ is $F(0,\cdot)\Psi$, then $Y$ gives a shift-coupling of $\Psi$ and its Palm version; i.e. $\theta_{Y}\Psi$ has the same distribution as the Palm version of $\Psi$.
\item
\label{thm:extraHead:2}
On the probability space $(\Omega,\mathcal F,\mathbb P_{\Psi})$, if $Y$ is a random vector such that its conditional distribution given $\Psi$ is \verified{$F(\cdot,0)\Phi$,}
then the random measure $(\mathbb P_{\Psi},\theta_{Y}\Psi)$ has the same distribution as $(\mathbb P,\Psi)$. In words, $\theta_{Y}\Psi$ is a reconstruction of $\Psi$ from its Palm version.
 \end{enumerate}
\end{theorem}


\verified{
Theorem~\ref{thm:extraHead} is a direct {implication} of Corollary~4.7 in~\cite{ThLa09} (see also Theorem~16 in~\cite{HoPe05} for point processes). {Nevertheless}, we state it for the purpose of this paper.
{In part~\eqref{thm:extraHead:1}, 
we call the random vector $Y$ an \defstyle{extra head scheme} for $\Psi$ with the terminology of~\cite{HoPe05}.}
One can also replace $F(0,\cdot)\Psi$ by $T(0,\cdot)$, where $T$ is a flow-adapted transport kernel which is almost surely $(\Phi,\Psi)$-balancing.
\verified{If moreover $T$ is a balancing allocation} {that is a function of $\Psi$,} provided that it exists, then $Y$ will be a deterministic vector conditional on $\Psi$. This is called a \defstyle{non-randomized extra head scheme} {(note that to construct $Y$ in the general case, one may need extra randomness; i.e. to extend the probability space)} and its existence is proved in the case when $\Psi$ is an ergodic simple point process in~\cite{HoPe06} {by providing a construction}. Algorithm~\ref{alg:Gale} is a generalization of \verified{that} construction. 
\verified{Although it does not provide a non-randomized extra head scheme in Theorem~\ref{thm:extraHead}, it has the property that the distribution of $Y$ conditional on $\Psi$ depends only on the realization of $\Psi$.} We don't know whether non-randomized extra head schemes always exist for general random measures or not, as stated in Open Problem~\ref{conj:non-randomized}. However, we show that
  a flow-adapted balancing allocation may not exist for \verified{two} general random measures, as shown in examples~\ref{ex:sqrt2} \verified{and~\ref{ex:noAllocation}}. \verified{Some results and open problems 
about {this problem}
 are provided in Subsection~\ref{sec:allocations}}.
}

The claim of Theorem~\ref{thm:extraHead} is similar to the  \myemph{inversion formula} (\cite{Me67}, Satz~2.4. See also~\cite{ThLa09}, (2.6)).  This formula recovers the distribution of $\Psi$ from its Palm distribution. Here, the difference is that the balancing property of $F$ ensures that the recovery can be done using a shift-coupling as defined in the theorem. 



 \begin{remark}
\label{rem:choice}
In Algorithm~\ref{alg:Gale}, one could define $A_n(x_0,\cdot)$ on $\partial \ball{x_0}{a_n(x_0)}$ and $R_{n}(\cdot,\xi_0)$ on $\partial \ball{\xi_0}{r_n(\xi_0)}$ in other ways. If this is done such that  Lemma~\ref{lemma:increasing} and Remark~\ref{rem:GaleShapley}  hold, then all of our results remain valid.
\end{remark}

\begin{remark}
Being stationary is crucial in Theorem~\ref{thm:coupling}. 
\verified{As an example, Example~\ref{ex:ZR} shows cases in which the site-optimal density is not balancing although
the (non-random) measures in the example may have equal \myemph{spatial intensities};} i.e. $\varphi([-r,r]) \sim \psi([-r,r])$ as $r\rightarrow\infty$ (see also{~\cite{HoPe06}}). 
{As mentioned in~\cite{HoPe-Tail}, it seems difficult to give a  sufficient condition for \verified{deterministic} discrete \verified{sets} \verified{(and therefore, for measures)} to ensure that the site-optimal density is balancing.}
\end{remark}
\verified{
\subsection{
Stability of \verified{Constrained} \verified{Densities}
}
\label{sec:stability}
}

 \begin{definition}
\label{def:stable}
Let \verified{$\varphi,\psi\in M$ be given and} $f$ be a given \verified{constrained} density. We say that a site $x_0$ is \defstyle{$f$-exhausted} if \[\int_{\myXi}{f(x_0,\xi)\psi(d\xi)}=1\] and \defstyle{$f$-unexhausted} otherwise. Similarly, a center $\xi_0$ is \defstyle{$f$-sated} if \[\int_{{\myX}}{f(x,\xi_0)\varphi(dx)}=1\] and \defstyle{$f$-unsated} otherwise. We say that a site $x_0$, \defstyle{$f$-desires} a center $\xi_0$ if $f(x_0,\xi_0)<1$ and either $x_0$ is $f$-unexhausted or
\begin{equation*}
\exists \xi_1\in \myXi: \dist{x_0}{\xi_1} > \dist{x_0}{\xi_0} \text{ and } f(x_0,\xi_1)>0.
\end{equation*}

 Similarly, we say $\xi_0$, \defstyle{$f$-desires} $x_0$ if $f(x_0,\xi_0)<1$ and either $\xi_0$ is $f$-unsated or
\begin{equation*}
\exists x_1\in {\myX}: \dist{x_1}{\xi_0}>\dist{x_0}{\xi_0} \text{ and } f(x_1,\xi_0)>0.
\end{equation*}

 We drop the prefix '$f$-' when there is no confusion.
A \verified{constrained} density $f$ is called \defstyle{stable} if there is no $(x_0,\xi_0)\in {\myX}\times \myXi$ such that $x_0$ desires $\xi_0$ and $\xi_0$ desires $x_0$.
\end{definition}

 In this definition, each sites prefers the centers {according to Euclidean distance} and vice-versa.
{See examples~\ref{ex:ZR} and~\ref{ex:Z*R} in Section~\ref{sec:examples} for examples of a stable and an unstable \verified{constrained} density.}

\verified{
}


\begin{theorem}[Stability]
\label{thm:stable}
The {site-optimal density} 
is a stable \verified{constrained} density \verified{and depends on $\varphi$ and $\omega$ in a measurable and flow-adapted manner.} 
\end{theorem}

\begin{theorem}
\label{thm:coupling}
Let $\Phi$ and $\Psi$ be ergodic random measures on $\mathbb R^d$ with positive and finite intensities $\lambda_{\Phi}$ and $\lambda_{\Psi}$.
Let $F=F_{\omega}(x,\xi)$ be a flow-adapted function \verified{which} is almost surely a stable \verified{constrained} density for $\Phi_{\omega}$ and $\Psi_{\omega}$. \verified{Then almost surely} 

 \begin{enumerate}[(i)]
 \item
\label{thm:coupling:=}
if $\lambda_{\Phi}=\lambda_{\Psi}$, then \verified{$F_{\omega}$ is $(\Phi_{\omega},\Psi_{\omega})$-balancing; i.e.} the set of unexhausted sites has zero $\Phi_{\omega}$-measure and the set of unsated centers has zero $\Psi_{\omega}$-measure. 
\item
\label{thm:coupling:<}
if $\lambda_{\Phi}<\lambda_{\Psi}$, then there is no unexhausted site but the set of unsated centers has an infinite $\Psi_{\omega}$-measure.
\item
\label{thm:coupling:2}
if $\lambda_{\Phi}>\lambda_{\Psi}$, then there is no unsated center but the set of unexhausted sites has an infinite $\Phi_{\omega}$-measure.
\end{enumerate}
\end{theorem}

 Proposition~\ref{prop:spatialAverage} quantifies how far the centers (resp. sites) are from being sated (resp. exhausted) in average in the second (resp. third) case of Theorem~\ref{thm:coupling}.
 
 \verified{
 \begin{remark}
 	 Theorems~\ref{thm:site-optimalBalancing}, \ref{thm:extraHead} and~\ref{thm:coupling} can be extended to stationary non-ergodic cases.  	 
 	 The necessary and sufficient condition on $\Phi$ and $\Psi$ is the equality of the \textit{sample intensities}
 	 \new{$\lim_{r\rightarrow\infty} {\Phi(\ball 0r)}/{\mathcal L(\ball 0r)}$ and $\lim_{r\rightarrow\infty} {\Psi(\ball 0r)}/{\mathcal L(\ball 0r)}$.
 	  This condition is proved to be necessary and sufficient for the existence of a flow-adapted balancing transport kernel in Theorem~5.1 in~\cite{ThLa09}.}
	\new{To do so}, we can slightly change the proofs and use conditional expectation with respect to the invariant sigma-filed (using the ideas in~\cite{ThLa09}) or obtain the general versions simply by applying the ergodic decomposition theorem.
\end{remark}
}

\subsection{\verified{Other Properties of Stable \verified{Constrained} {densities}}}
\label{sec:other}
\verified{ In this subsection we study monotonicity, optimality, uniqueness and boundedness of territories of stable \verified{constrained} densities.}
In general, uniqueness is not granted in the deterministic case (see~\cite{HoPe06} for a counter example). The manner of choosing between equidistant points in Algorithm~\ref{alg:Gale} (see Remark~\ref{rem:choice}) is another obstacle for uniqueness;
e.g. when $\varphi$ and $\psi$ are measures on $\mathbb Z^d$.
Assumption~\ref{assump:uniqueness} gives a sufficient condition for uniqueness of the choice in Remark~\ref{rem:choice} for almost all points. {However, we will prove uniqueness of stable \verified{constrained} densities only in the stationary case in Theorem~\ref{thm:uniqueness}.} {\verified{It can be seen that} the first condition in Assumption~\ref{assump:uniqueness} means that, for $\varphi$-a.e. site $x$, the boundary of no ball centered at $x$ can be partitioned in two disjoint sets with positive $\psi$-measure.}
This assumption is not difficult to satisfy as shown by Proposition~\ref{prop:assumption}.

\begin{assumption}
\label{assump:uniqueness}
For $\varphi$-a.e. site $x$ and $\psi$-a.e. center $\xi$, we have
\begin{eqnarray}
\label{eq:assump:uniqueness1}
\forall r>0, &\exists s\in\partial \ball xr:& \psi(\partial\ball xr \backslash \{s\})=0,\\
\label{eq:assump:uniqueness2}
\forall r>0, &\exists s\in\partial \ball {\xi}r:& \varphi(\partial\ball {\xi}r \backslash \{s\})=0.
\end{eqnarray}

 \end{assumption}

 \begin{proposition}
\label{prop:assumption}
If at least one of $\varphi$ and $\psi$  \verified{assigns zero to all spheres and all affine hyperplanes of $\mathbb R^d$; e.g.} is absolutely continuous with respect to the Lebesgue measure, then Assumption~\ref{assump:uniqueness} holds. 
\end{proposition}


{The following auxiliary functions measure how far the mass is transported from sites and to centers, given a \verified{sub-balancing} density $f$.}
\begin{definition}
\label{def:g,h}
Let $f$ be a \verified{sub-balancing} density 
as in Definition~\ref{def:compatibleDensity}. For $t\in[0,\infty]$ define
\begin{eqnarray*}
\g xft{\psi} &:=& \int_{\mathbb R^d}f(x,\xi)1_{\abs{x-\xi}\leq t}\psi(d\xi),
\\
\h {\xi}ft{\varphi} &:=& \int_{\mathbb R^d} f(x,\xi)1_{\abs{x-\xi}\leq t}\varphi(dx).
\end{eqnarray*}
\end{definition}
\begin{remark}
\label{rem:desire}
If $f(x,\xi)<1$ and $\g xf{\dist x{\xi}}{\psi} <1$, then $x$ desires $\xi$. Similarly, if $f(x,\xi)<1$ and $\h {\xi}f{\dist{x}{\xi}}{\varphi}<1$, then $\xi$ desires $x$. {Note that the converse is not true; i.e. $x$ may desire $\xi$ even if $\g xf{\dist x{\xi}}{\psi}=1$ since $f(x,\cdot)$ can be positive on a $\psi$-null set that contains a center farther than $\xi$ to $x$.}
\end{remark}

 \begin{proposition}
\label{prop:spatialAverage}
In the setting of Theorem~\ref{thm:coupling}, one has
\begin{eqnarray*}
\lim_{r\rightarrow\infty}\frac 1{\Phi(\ball 0r)}\int_{\ball 0r} \g xF{\infty}{\Psi} \Phi(dx) &=& \min\{1,\frac{\lambda_{\Psi}}{\lambda_{\Phi}}\}, \text{a.s.}\\
\lim_{r\rightarrow\infty} \frac 1{\Psi(\ball 0r)} \int_{\ball 0r} \h {\xi}F{\infty}{\Phi}\Psi(d\xi) &=& \min\{1,\frac{\lambda_{\Phi}}{\lambda_{\Psi}}\}, \text{a.s.}
\end{eqnarray*}
\end{proposition}
\verified{Note that the left hand sides in the above equations are} the \myemph{spatial averages} of $\g xF{\infty}{\Psi}$ and $\h {\xi}F{\infty}{\Phi}$.
{This is a measure of how far the sites and the centers are from being satisfied}.

 \begin{theorem}[Monotonicity]
\label{thm:monotonicity}
Let $(\varphi,\psi)$ and $(\mu,\nu)$ be two pairs of measures such that $\mu\geq\varphi$, $\nu\leq\psi$ and the pair $(\mu,\psi)$ satisfies Assumption~\ref{assump:uniqueness}.
Let $f$ be an arbitrary stable \verified{constrained} density for $(\mu,\nu)$ and consider the site-optimal density $f_s$ for $(\varphi,\psi)$ together with the functions in Definition~\ref{def:site-optimalDensity}.
\begin{enumerate}[(i)]
\item
\label{item:monotonicity1}
We have
\[
f+R\leq 1,\quad (\mu\otimes\psi)\text{-a.e.}
\]
\item
\label{item:monotonicity2}
For $(\mu\otimes\psi)$-a.e. $(x,\xi)$, if $x$ fully applies to $\xi$, for example if $\abs{x-\xi}<a(x)$, then $f(x,\xi)\leq f_s(x,\xi)$.
\item
\label{item:monotonicity3}
For $\mu$-a.e. site $x$ we have
\[
\g x{f}t{\nu} \leq \g x{f_s}t{\psi},\quad \forall t\in[0,\infty].
\]
\item
\label{item:monotonicity4}
For $\psi$-a.e. center $\xi$ we have
\[
\h {\xi}{f}t{\mu} \geq \h{\xi}{f_s}t{\varphi},\quad\forall t\in [0,\infty].
\]
\end{enumerate}
\end{theorem}

 Intuitively, it means that when there are less sites and more centers, the situation is better for sites and worse for centers.
The following corollary is immediately obtained from Theorem~\ref{thm:monotonicity}. Note that the inequalities that contain $f_c$ {are equivalent to the ones that contain $f_s$ as seen by} swapping the roles of the sites and the centers. 
\begin{corollary}[Optimality]
\label{cor:optimality}
Suppose $\varphi$ and $\psi$ satisfy Assumption~\ref{assump:uniqueness} and let $f$ be an arbitrary stable \verified{constrained} density for $\varphi$ and $\psi$. Let $f_s$ and $f_c$ be the site-optimal and the center-optimal densities for the same measures.
\begin{enumerate}[(i)]
\item
For $\varphi$-a.e. site $x$ we have
\[\g x{f_s}t{\psi} \geq \g xft{\psi} \geq \g x{f_c}t{\psi}, \quad\forall t\in[0,\infty].\]
\item
For $\psi$-a.e. center $\xi$ we have
\[\h {\xi}{f_c}t{\varphi} \geq \h {\xi}ft{\varphi} \geq \h {\xi}{f_s}t{\varphi}, \quad\forall t\in[0,\infty].\]
\end{enumerate}
\end{corollary}

 In words, among all stable \verified{constrained} densities, the site-optimal density is the best for sites and the worst for centers. This justifies the names \myemph{site-optimal} and \myemph{center-optimal} for the Gale-Shapley algorithm.

As mentioned in~\cite{HoPe-Tail}, it seems difficult to express a simple condition in terms of $\varphi$ and $\psi$ that ensures uniqueness of stable \verified{constrained} densities. But Assumption~\ref{assump:uniqueness} is enough in the stationary case, as shown in Theorem~\ref{thm:uniqueness}. The key for proving the theorem is the following proposition. 

 \begin{proposition}
\label{prop:uniqueness}
With the assumptions of Corollary~\ref{cor:optimality}, if for $\varphi$-a.e. site $x$ we have
\[
\g x{f_s}t{\psi} = \g x{f_c}t{\psi},\quad \forall t\in[0,\infty],
\]
then there is a $(\varphi\otimes\psi)$-a.e. unique stable \verified{constrained} density for $\varphi$ and $\psi$.
\end{proposition}

 Here, by $(\varphi\otimes\psi)$-a.e. unique, we mean that any two stable \verified{constrained} densities are identical except on a set of zero $(\varphi\otimes\psi)$-measure.

\begin{theorem}[Uniqueness]
\label{thm:uniqueness}
Let $\Phi$ and $\Psi$ be stationary random measures on $\mathbb R^d$ with positive finite intensities that satisfy Assumption~\ref{assump:uniqueness} almost surely.
{Almost surely,} any two stable \verified{constrained} densities for $\Phi_{\omega}$ and $\Psi_{\omega}$ agree on $(\Phi_{\omega}\otimes\Psi_{\omega})$-a.a. points.
\end{theorem}

 Here, we mean that there is an event with probability one such that uniqueness holds on that event.

\verified{
 Given a stable \verified{constrained} density $f$, the \defstyle {territory} of a site $x$ is the  set $\{\xi\in\mathbb R^d: f(x,\xi)>0\}$. Similarly, the territory of a center $\xi$ is the set $\{x\in\mathbb R^d: f(x,\xi)>0\}$.
}


 \begin{theorem}[Boundedness]
\label{thm:bounded}
Let $\Phi$ and $\Psi$ be stationary random measures on $\mathbb R^d$ {which are almost surely non-zero and have} finite intensities. Let $F=F_{\omega}(x,\xi)$ be a flow-adapted stable \verified{constrained} density for $(\Phi,\Psi)$. Almost surely we have
\begin{enumerate}[(i)]
\item
\label{thm:bounded:1}
$\Phi_{\omega}$-a.a. sites and $\Psi_{\omega}$-a.a. centers have bounded territories.
\item
\label{thm:bounded:2}
The union of the territories of the sites (resp. centers) in a bounded set, has finite $\Psi$-measure (resp. finite $\Phi$-measure).
\end{enumerate}
\end{theorem}


\subsection{\verified{Voronoi Transport Kernel {Corresponding to a Measure}}}
\label{sec:voronoi}
\verified{
In this subsection, we generalize the notion of Voronoi tessellation and define it for a measure. We will use it only for proving Theorem~\ref{thm:bounded} here, but it can be interesting in its own.}

 \begin{definition}
 \label{def:vtk}
Let $\psi$ be a measure on $\mathbb R^d$ \verified{such that $\psi(\mathbb R^d)\geq 1$}. \verified{For $x_0\in\mathbb R^d$ define} 
\begin{eqnarray*}
s(x_0):=\sup\{s\in\mathbb R:\psi(\ball {x_0}s) \leq 1\}.
\end{eqnarray*}
\verified{For $x_0,\xi\in\mathbb R^d$ define}
\begin{eqnarray*}
v(x_0,\xi)&:=& \left\{
\begin{array}{ll}
1, & \dist {x_0}{\xi}<s(x_0),\\
c, & \dist {x_0}{\xi}=s(x_0),\\
0, & \dist {x_0}{\xi}>s(x_0),
\end{array}\right.
\end{eqnarray*}
\verified{
where $c$ is the constant in $[0,1]$ such that $\int v(x_0,\xi)\psi(d\xi)=1$. 
We let \verified{$c=1$} if $s(x_0)=\infty$ or $\psi(\partial \ball {x_0}{s(x_0)})=0$.}
The \defstyle{\vdensity} and the \defstyle{\vtk{}} with respect to $\psi$ are the function $v$ and the transport kernel $V(x,\cdot):= v(x,\cdot)\psi$ respectively. The \defstyle{\vterritory} of center $\xi$ with respect to $\psi$ is \verified{the set $\{x\in\mathbb R^d: v(x,\xi)>0\}$.} 
\end{definition}

 {Note that the \vdensity{} is {the same} as the function $A_1$ in Algorithm~\ref{alg:Gale}.} 

 \begin{remark}
 \verified{The condition $\psi(\mathbb R^d)\geq 1$ ensures that}
the \vtk{} is a non-weighted transport kernel. If $\psi$ is the counting measure on a discrete set $\Xi$, then
the \vterritories{} with respect to $\psi$ give the usual Voronoi tessellation of $\Xi$. Therefore, the notion of \vtk{} generalizes the notion of Voronoi tessellation.
\end{remark}

\verified{
\begin{remark}
	\label{rem:VoronoiIff}
 	 It is easy to see that $x$ is in the Voronoi territory of $\xi$ if and only if  $\psi(\ballint{x}{\dist{x}{\xi}}) \leq 1$ and if equality happens, then $\psi(\partial \ball{x}{\dist{x}{\xi}}) =0$.
\end{remark}
}
 \begin{lemma}
\label{lemma:star}
The \vterritory{} of $\xi$ is star-shaped with center $\xi$ but not necessarily convex \verified{or closed. Moreover, it is a closed polyhedral region in $\mathbb R^d$ provided that $\psi$ has a discrete support.} 
\end{lemma}

 \begin{proposition}
\label{prop:vterritories}
\verified{If every half-space has infinite $\psi$-measure, then all \vterritories{} with respect to $\psi$  are bounded. In particular,}
if $\Psi$ is a stationary random measure {which is almost surely non-zero,} then almost surely all \vterritories{} with respect to $\Psi$ are bounded.
\end{proposition}

 \subsection{\verified{On Allocations}}
\label{sec:allocations}
\verified{
As explained in the introduction, a problem of special interest is constructing flow-adapted balancing allocations.
But since a \verified{$\psi$-constrained} {weighted} transport kernel $T$ satisfies $T\leq \psi$ (as in Definition~\ref{def:mildTransport}), a balancing \verified{$\psi$-constrained} transport kernel cannot be an allocation except maybe when $\psi$ is a discrete measure. The reason is, if $B$ is a set of centers such that $0<\psi(B)<1$ (which exists in the non-discrete case), then on a set of sites with positive measure (under $\varphi$) we have $T(x,B)>0$. On the other hand, $T(x,B)<1$ and so $T(x,\cdot)$ is not a Dirac measure. However, \verified{we will see that} under some conditions in the case when $\psi$ is a counting measure, the site-optimal density is guarantied to \verified{give} an allocation.  

Motivating from\verified{~\cite{HoPe06} and} Definition~\ref{def:stable}, we define

\verified{
\begin{definition}
	\label{def:allocation}
Given $\varphi,\psi\in M$,
an \defstyle{allocation} is a measurable function $\tau:D\rightarrow \mathbb R^d\cup\{\infty\}$, where $D\subseteq\mathbb R^d$ is a measurable set such that $\varphi(D^c)=0$ and $\infty$ is treated as a single point added to the space. This can be regarded a weighted transport kernel that transports the mass at $s$ to a single point $\tau(s)$; i.e. define $T(s,\cdot)$ to be $\delta_{\tau(s)}$ whenever $s\not\in D \cup \tau^{-1}(\infty)$ and zero otherwise.  This allocation is {$(\varphi,\psi)$-balancing} if $\tau\neq \infty, \varphi$-a.e. and $\varphi(\tau^{-1}(\cdot))=\psi(\cdot)$.
\end{definition}
}

\begin{definition}
\label{def:stableAllocation}
Given measures $\varphi$ and $\psi$ on $\mathbb R^d$, let $\tau$ be an allocation defined on a domain $D\subseteq \text{supp}(\varphi)$ \verified{with $\varphi(D^c)=0$}. We say $\tau$ is \defstyle{sub-balancing} when $\varphi(\tau^{-1}(B))\leq \psi(B)$ for all $B\in\mathcal G$; i.e. $\tau_*(\varphi|_D)\leq\psi$. Given a sub-balancing allocation $\tau$, a site \verified{$x_0\in D$} is \defstyle{exhausted} when \verified{$\tau(x_0)\neq\infty$ and \defstyle{unexhausted} when $\tau(x_0)=\infty$.} A center $\xi_0$ is \defstyle{sated} when it is not in the support of the measure $\psi-\tau_*(\varphi|_D)$. we say a site $x_0$ \defstyle{desires} a center $\xi_0$ when either $x_0$ is unexhausted or  $\dist{x_0}{\tau(x_0)}<\dist{x_0}{\xi_0}$. We say $\xi_0$ \defstyle{desires} $x_0$ if $\tau(x_0)\neq \xi_0$ and either $\xi_0$ is unsated or there is a sites $x_1$ s.th. $ \tau(x_1)=\xi_0$ and $\dist{x_1}{\xi_0}>\dist{x_0}{\xi_0}$.  We say that a sub-balancing allocation is \defstyle{stable} if 
there is no pair $(x_0,\xi_0)$ such that both desire each other.
\end{definition}


A similar notion of one-sided stable allocations in dimension one is also defined in~\cite{LaMoTh14}. The authors construct a balancing stable allocation (in that sense) between two {ergodic} stationary and mutually singular diffuse (i.e. without atom) random measures on $\mathbb R$.

\begin{proposition}
\label{prop:01}
In the following cases the site-optimal density is $\{0,1\}$-valued on $(\varphi\otimes\psi)$-a.e. points.
\begin{enumerate}[(i)]
\item
\label{prop:01-1}
when both $\varphi$ and $\psi$ are absolutely continuous w.r.t. $\mathcal L$,
\item
\label{prop:01-2}
when one of $\varphi$ and $\psi$ is absolutely continuous w.r.t. $\mathcal L$ and the other is a counting measure,
\item
\label{prop:01-3}
when both $\varphi$ and $\psi$ are counting measures and Assumption~\ref{assump:uniqueness} holds.
\end{enumerate}
\end{proposition}

We skip the proof of Proposition~\ref{prop:01}. For a proof, it is easy to use induction to show that the functions in Algorithm~\ref{alg:Gale} are $\{0,1\}$-valued a.e.

\begin{remark}
\label{rem:01}
In the case that each one of $\varphi$ and $\psi$ is either a diffuse (i.e. without atom) measure or a counting measure (not necessarily satisfying Assumption~\ref{assump:uniqueness}), Algorithm~\ref{alg:Gale} can be slightly changed to find a $\{0,1\}$-valued {stable {constrained} density}. It is enough to re-define the functions $A_n(x_0,\cdot)$ and $R_n(\cdot,\xi_0)$ on $\partial \ball{x_0}{a_n(x_0)}$ and $\partial\ball{\xi_0}{r_n(\xi_0)}$ for example using the lexicographic order on the boundary of the balls (see Remark~\ref{rem:choice}).
\end{remark}

\begin{proposition}
A \verified{constrained} density which is $\{0,1\}$-valued on $(\varphi\otimes\psi)$-a.e. points gives an allocation provided that $\psi$ is a counting measure. Therefore, by Remark~\ref{rem:01}, for the existence {(and construction)} of stable allocations it is enough that
\begin{enumerate}[(i)]
\item
$\varphi$ is diffuse and $\psi$ is a counting measure, 
\item
or both $\varphi$ and $\psi$ are counting measures.
\end{enumerate}
\end{proposition}

Here, by \textit{`gives an allocation'} we mean that the \verified{weighted} transport kernel given by Remark~\ref{rem:kernel} coincides with an allocation on almost all sites.


\begin{remark}
When $\psi$ is {a measure with discrete support } $\psi=\sum_{i=1}^{\infty}w_i\delta_{\xi_i}$ and $\varphi$ is diffuse, we can slightly change 
Algorithm~\ref{alg:Gale} to obtain a {stable allocation}.
 \verified{We limit ourselves to $\bar{\psi}$-constrained transport kernels,}
where $\bar{\psi}=\sum_{i=1}^{\infty}\delta_{\xi_i}$ is the counting measure with the same support as $\psi$. Equivalently, $f(x,\xi_i)\leq \frac 1{w_i}$ for defining \verified{constrained} densities. We then change the definition of $A_n$ and $a_n$ in Algorithm~\ref{alg:Gale} in a similar manner. 
We similarly get the following proposition.
\end{remark}

\begin{proposition}
Suppose $\Phi$ is a diffuse random measure and $\Psi$ is {a random measure whose support is a.s. discrete} in the setting of Theorem~\ref{thm:site-optimalBalancing}. One can construct an allocation which is almost surely stable and $(\Phi,\Psi)$-balancing.
\end{proposition}

{We finish this section} with two open problems.

\begin{openproblem}
\label{conj:stable}
Does there exist  a stable allocation for any two measures $\varphi$ and $\psi$ on $\mathbb R^d$ provided that $\varphi$ is absolutely continuous with respect to the Lebesgue measure?
\end{openproblem}

\new{
If a stable allocation exists and can be chosen as a measurable function of $(\varphi,\psi)$, then it leads to a solution of the following question in a way similar to theorems~\ref{thm:site-optimalBalancing} and~\ref{thm:extraHead}.
 }

\begin{openproblem}
\label{conj:non-randomized}
Does
there exists a non-randomized {extra head scheme} for any ergodic random measure on $\mathbb R^d$ with a positive and finite intensity?  
\end{openproblem}

}


 \section{Proofs}
\label{sec:proofs}

\begin{proof}[Proof of Lemma~\ref{lemma:infinite}]
	\new{
		Assume $\mathbb P[\Phi(V)<\infty]=2\delta>0$ an let $\epsilon>0$ be arbitrary. This implies that there exists $R>0$ such that $\mathbb P[\Phi(V\backslash B_R(0))<\epsilon]>\delta$. For a given $r>0$, let $\tilde{B_r}\subseteq V\backslash B_R(0)$ be a ball with radius $r$, which exists by the assumption on $V$. By stationarity we get
		$$\mathbb P[\Phi( B_r(0))<\epsilon]=\mathbb P[\Phi(\tilde{B_r})<\epsilon]\geq\mathbb P[\Phi(V\backslash B_R(0))<\epsilon]>\delta.$$
		By letting $r\rightarrow\infty$ we get $\mathbb P[\Phi(\mathbb R^d)\leq\epsilon]\geq\delta$. Since $\epsilon$ is arbitrary, this gives $\mathbb P[\Phi(\mathbb R^d)=0]\geq\delta$, a contradiction.
	}
	
\end{proof}

 \begin{remark}
\label{rem:GaleShapley}
By the definition of $A_n$ we have
\[
\int_{\myXi} \left(A_n(x_0,\xi)-R_{n-1}(x_0,\xi)\right)\psi(d\xi)\leq 1
\]
and if $a_n(x_0)<\infty$, equality holds. Moreover, if $\psi$ has infinite total mass, then one can prove by induction that $a_n(\cdot)<\infty$ for all $n$. Similarly, by the definition of $R_{n}$ we have
\[
\int_{\myX} \left(A_n(x,\xi_0)-R_{n}(x,\xi_0)\right)\varphi(dx)\leq 1
\]
and if $r_n(\xi_0)<\infty$, equality holds. Moreover, if equality holds 
at some stage $n$, then it holds at all later stages; i.e. $\xi_0$ is sated at stage $n$ afterwards.
\end{remark}

\begin{proof}[Proof of Lemma~\ref{lemma:increasing}]
It is clear that $R_1\geq R_0$. For $n\geq 2$, assume $R_{n-1}\geq R_{n-2}$. We will conclude $a_n\geq a_{n-1}$, $A_n\geq A_{n-1}$, $r_n\leq r_{n-1}$ and $R_{n}\geq R_{n-1}$, which proves the claim by induction. \verified{By the same induction, it is easily proved that the functions depend on $\varphi$ and $\psi$ in a measurable and flow-adapted manner.}

 Since $1-R_{n-1}\leq1-R_{n-2}$,~\eqref{eq:a_n} gives $a_n(x_0)\geq a_{n-1}(x_0)$.
If we have $a_n(x_0)>a_{n-1}(x_0)$ or $a_n(x_0)=\infty$, then it is clear from the definition of $A_n$ that $A_n(x_0,\cdot)\geq A_{n-1}(x_0,\cdot)$. Now, suppose $a_n(x_0)=a_{n-1}(x_0)=:t<\infty$.
By~\eqref{eq:a_n} \verified{and~\eqref{eq:c_n}} we get
\[
c_{n}(x_0)=1-\frac{1-\int_{B^{\circ}}(1-R_{n-1}(x_0,\xi))\psi(d\xi)}{\int_{\partial B} (1-R_{n-1}(x_0,\xi))\psi(d\xi)},
\]
where $B=\ball{x_0}t$ and \verified{$\frac 00=0$} by convention. This equation and the fact that $R_{n-1}\geq R_{n-2}$ implies $c_n(x_0)\leq c_{n-1}(x_0)$ and hence $A_n(x_0,\cdot)\geq A_{n-1}(x_0,\cdot)$.
Given $A_n\geq A_{n-1}$, \eqref{eq:r_n} implies $r_n(\xi_0)\leq r_{n-1}(\xi_0)$. The proof of the fact that $r_n\leq r_{n-1}$ implies $R_{n}\geq R_{n-1}$ is completely similar to the above proof. 

 \end{proof}

\begin{lemma}
\label{lemma:unexhaustedApply}
In the site-optimal density, if a site $x_0$ desires a center $\xi_0$, then $A_n(x_0,\xi_0)=1$ for sufficiently large $n$. Similarly, if $\xi_0$ desires $x_0$ then $R_n(x_0,\xi_0)=0$ for all $n$.

 \end{lemma}

 \begin{proof}
Suppose $x_0$ desire $\xi_0$.
If there is a center $\xi_1$ such that $\abs{\xi_1-x_0}>\abs{\xi_0-x_0}$ and $f(x_0,\xi_1)>0$, then $x_0$ has applied to $\xi_1$ and so it has applied to all points closer than $\xi_1$ with weight $1$. Hence $x_0$ has applied to $\xi_0$ with weight $1$. In the other case, $x_0$ is unexhausted.
It is enough to prove that $a_n(x_0)>\dist {x_0}{\xi_0}$ for sufficiently large $n$. If this is not true, $a_n(x_0)<\infty$ for all $n$ and Remark~\ref{rem:GaleShapley} gives
\[\int_{\myXi} \left(A_n(x_0,\xi)-R_{n-1}(x_0,\xi)\right)\psi(d\xi)= 1.\]
Since $a_n(x_0)<\dist{x_0}{\xi_0}$, $A_n(x_0,\cdot)-R_{n-1}(x_0,\cdot)$ is bounded by {$\identity_{\ball{x_0}{\dist{x_0}{\xi_0}}}$,} which is integrable with respect to $\psi$ due to locally finiteness of $\psi$. So, Lebesgue's dominated convergence theorem gives $\int_{\myXi}f_s(x_0,\xi)\psi(d\xi)=1$; i.e. $x_0$ is exhausted,
a contradiction.

Now, suppose that $\xi_0$ desires $x_0$. If $\xi_0$ is unsated, then it has not rejected any weights. So suppose there is a site $x_1$ such that $\abs{x_1-\xi_0}>\abs{x_0-\xi_0}$ and $f(x_1,\xi_0)>0$. Therefore $r_n(\xi_0)\geq \abs{x_1-\xi_0}$ for all $n$, or else $\xi_0$ would fully reject $x_1$ at all stages after stage $n$ since the rejection radius is non-increasing. Thus $r_n(\xi_0)>\abs{x_0-\xi_0}$ for all $n$ and so $\xi_0$ has not rejected any application weight of $x_0$.
\end{proof}

 \begin{proof}[Proof of Theorem~\ref{thm:stable}]
Let $f_s$ be the site-optimal density. \verified{By Lemma~\ref{lemma:increasing} we get that $f_s$ depends on $\varphi$ and $\omega$ in a measurable and flow-adapted manner.} We have \[f_s=\lim_{n\rightarrow\infty}A_n-R_{n-1}=\lim_{n\rightarrow\infty}A_n-R_{n}.\] Moreover, $A_n\geq R_{n-1}$ and $A_n\geq R_{n}$. By Remark~\ref{rem:GaleShapley} and Fato's lemma, it follows that $f_s$ is a \verified{sub-balancing} density. Also, $f_s$ is \verified{constrained} since $f_s\leq 1$. Now, suppose that $f_s$ is unstable. So we can find a site $x_0$ and a center $\xi_0$ that desire each other. By Lemma~\ref{lemma:unexhaustedApply} we get that $A_n(x_0,\xi_0)=1$ and $R_n(x_0,\xi_0)=0$ for sufficiently large $n$. So $f_s(x_0,\xi_0)=1-0=1$, a contradiction.

 \end{proof}

\begin{lemma}
\label{lemma:satedOrExhausted}
Let $f$ be any stable \verified{constrained} density for $\varphi$ and $\psi$ as in Definition~\ref{def:stable}.
Let $X'$ be the set of unexhausted sites and $\Xi'$ be the set of unsated centers. If $X'\neq \emptyset$, then $\psi(\Xi')<1$ and if $\Xi' \neq\emptyset$, then $\varphi(X')<1$. In particular, we have either $\varphi(X') < 1$ or $\psi(\Xi') < 1$.
\end{lemma}

\begin{proof}
First, suppose $X'\neq\emptyset$ and $\psi(\Xi')\geq 1$. Let $x_0\in X'$. Since $x_0$ is unexhausted and $\psi(\Xi')\geq 1$, we find a point $\xi_0\in \Xi'$ such that $f(x_0,\xi_0)<1$. Now $(x_0,\xi_0)$ is an unstable pair since $x_0$ is unexhausted and $\xi_0$ is unsated, a contradiction. Similarly, if $\Xi'\neq\emptyset$ we conclude that $\varphi(X_0)<1$, which completes the proof.

 \end{proof}

 Note that it is possible that both $X'$ and $\Xi'$ in Lemma~\ref{lemma:satedOrExhausted} have positive measure, as shown in Example~\ref{ex:interval}.

 \begin{lemma}
\label{lemma:doubleCounting}
Let $\Phi$ and $\Psi$ be stationary random measures on $\mathbb R^d$ satisfying the assumptions of Theorem~\ref{thm:coupling}.
For any flow-adapted \verified{sub-balancing} density $F$ we have
\begin{equation}
\lambda_{\Phi} \omidPalm{\Phi}{\g 0Ft{\Psi}} = \lambda_{\Psi} \omidPalm{\Psi}{\h 0Ft{\Phi}}.
\end{equation}
\end{lemma}

Intuitively, this means that the average mass that is transported to {a typical center from the sites of distance at most $t$} is equal to the mass that is transported from {a typical site to the centers of distance at most $t$}. \verified{This is a version of \textit{mass transport principles}. A more general equation can be found in~\cite{ThLa09}.}

 \begin{proof}
Let $H(\omega,s):=F_{\omega}(0,s)1_{\abs{s}<t}$. By Neveu's exchange formula~\eqref{eq:Neveu} we get
\begin{eqnarray*}
\lambda_{\Phi} \omidPalm{\Phi}{\g 0Ft{\Psi}} &=& \lambda_{\Phi}\omidPalm{\Phi}{\int_{\mathbb R^d}F_{\theta_0}(0,s)1_{\abs{s}<t}\Psi(ds)}
\\
&=& \lambda_{\Psi} \omidPalm{\Psi}{\int_{\mathbb R^d} F_{\theta_s}(0,-s)1_{\norm{-s}<t} \Phi(d s)}\\
&=& \lambda_{\Psi} \omidPalm{\Psi}{\int_{\mathbb R^d} F_{\theta_0}(s,0)1_{\norm{s}<t} \Phi(d s)}\\
&=& \lambda_{\Psi} \omidPalm{\Psi}{\h 0Ft{\Phi}}.
\end{eqnarray*}
\end{proof}

 \begin{proof}[Proof of Theorem~\ref{thm:coupling}]
Let $\Phi_1$ and $\Psi_1$ be the restrictions of $\Phi$ and $\Psi$ to the set of unexhausted sites and unsated centers respectively. {Since $F$ is flow-adapted, $\Phi_1$ and $\Psi_1$ are also ergodic random measures.}
\verified{Using Definition~\ref{def:g,h},} define
\begin{eqnarray*}
U_x(F) &:=& 1 - \g xF{\infty}{\Psi},\\
U^{\xi}(F) &:=& 1- \h {\xi}F{\infty}{\Phi}.
\end{eqnarray*}
By Lemma~\ref{lemma:doubleCounting} for $t=\infty$ we get
\[
\lambda_{\Phi}\omidPalm{\Phi}{1-U_0(F)} =\lambda_{\Psi}\omidPalm{\Psi}{1-U^0(F)}.
\]
Thus
\begin{equation}
\label{eq:thm:coupling-1}
\lambda_{\Phi}\omidPalm{\Phi}{U_0(F)}-\lambda_{\Psi}\omidPalm{\Psi}{U^0(F)} = \lambda_{\Phi} - \lambda_{\Psi}.
\end{equation}
By \eqref{eq:omidPalm}, if $B$ is a cube in $\mathbb R^d$, then we have
\begin{eqnarray*}
\lambda_{\Phi}\omidPalm{\Phi}{U_0(F)} &=& \frac 1{\mathcal L_d(B)}\omid{\int_{B}U_0(F_{\theta_s})\Phi(ds)}\\
&=& \frac 1{\mathcal L_d(B)}\omid{\int_{B}U_s(F)\Phi(ds)}\\
&=& \frac 1{\mathcal L_d(B)}\omid{\int_{B}U_s(F)\Phi_1(ds)},
\end{eqnarray*}
{where the last equality is due to the fact that $\Phi_1 = \identity_{\{U_s(F)>0\}} \Phi$.}
Since $U_s(F)$ is bounded and $\Phi_1$ is ergodic, we get that
\begin{eqnarray}
\label{eq:thm:coupling-2}
\left\{
\renewcommand*{\arraystretch}{1.2}
\begin{array}{rcl}
\omidPalm{\Phi}{U_0(F)}=0 &\Leftrightarrow &\omid{\Phi_1(B)}=0 \Leftrightarrow \Phi_1(\mathbb R^d)=0, \text{a.s.}\\
\omidPalm{\Phi}{U_0(F)}>0 &\Leftrightarrow &\omid{\Phi_1(B)}>0 \Leftrightarrow \Phi_1(\mathbb R^d)=\infty, \text{a.s.}
\end{array}\right.
\end{eqnarray}
Similarly,
\begin{eqnarray}
\label{eq:thm:coupling-3}
\left\{
\renewcommand*{\arraystretch}{1.2}
\begin{array}{rcl}
\omidPalm{\Psi}{U^0(F)}=0 &\Leftrightarrow &\omid{\Psi_1(B)}=0 \Leftrightarrow \Psi_1(\mathbb R^d)=0, \text{a.s.}\\
\omidPalm{\Psi}{U^0(F)}>0 &\Leftrightarrow &\omid{\Psi_1(B)}>0 \Leftrightarrow \Psi_1(\mathbb R^d)=\infty, \text{a.s.}
\end{array}\right.
\end{eqnarray}
Since $F$ is stable almost surely, Lemma~\ref{lemma:satedOrExhausted} gives that, almost surely, either $\Phi_1(\mathbb R^d) < \infty$ or $\Psi_1(\mathbb R^d) < \infty$. By ergodicity, either $\Phi_1(\mathbb R^d)<\infty$ a.s. or $\Psi_1(\mathbb R^d)<\infty$ a.s. Hence, according to~\eqref{eq:thm:coupling-2} and~\eqref{eq:thm:coupling-3}, either $\omidPalm{\Phi}{U_0(F)}=0$ or $\omidPalm{\Psi}{U^0(F)}=0$. If we substitute this in~\eqref{eq:thm:coupling-1}, we get
\begin{equation}
\label{eq:thm:coupling-4}
\renewcommand*{\arraystretch}{1.4}
\left\{\begin{array}{lcl}
\omidPalm{\Phi}{U_0(F)} &=& \max\{0,1-\frac{\lambda_{\Psi}}{\lambda_{\Phi}}\},\\
\omidPalm{\Psi}{U^0(F)} &=& \max\{0,1-\frac{\lambda_{\Phi}}{\lambda_{\Psi}}\}.
\end{array}\right.
\end{equation}
These equalities, \eqref{eq:thm:coupling-2} and~\eqref{eq:thm:coupling-3} complete the proof.
\end{proof}

%

\verified{
\begin{proof}[Proof of Theorem~\ref{thm:site-optimalBalancing}]
By Lemma~\ref{lemma:increasing}, the site-optimal density for $\Phi$ and $\Psi$ gives a flow-adapted transport kernel. Moreover, it is a stable  density by Theorem~\ref{thm:stable}. Now the theorem is a direct consequence of part~\eqref{thm:coupling:=} of Theorem~\ref{thm:coupling}.
\end{proof}
}

\begin{proof}[Proof of Theorem~\ref{thm:extraHead}]

%
\verified{Let $H$ be any measurable function on $M$. Since $\mathbb P_{\Phi}=\mathbb P$, it is a direct consequence of Corollary~4.7 in~\cite{ThLa09} that $\omidPalm{\Psi}{H(\Psi)} = \omid{H(\theta_Y\Psi)}$ in case~\eqref{thm:extraHead:1} and $\omidPalm{\Psi}{H(\theta_Y\Psi)}=\omid{H(\Psi)}$ in case~\eqref{thm:extraHead:2}. This finishes the proof. In fact, case~\eqref{thm:extraHead:1} is just Example~4.8 in~\cite{ThLa09}.}

 \end{proof}

\begin{proof}[Proof of Proposition~\ref{prop:assumption}]
\verified{Suppose $\varphi$ assigns zero to all spheres and all affine hyperplanes.} So $\varphi(\partial B)=0$ for every ball $B$. Thus,~\eqref{eq:assump:uniqueness2} holds for all $\xi$.
For $0\leq k<d$, by a $k$-dimensional sphere, we mean the intersection of the boundary of a ball with a non-tangent affine subspace of dimension $k+1$. We call a $k$-dimensional sphere $S$ \myemph{bad} if $\psi(S \backslash \{s\})>0$ for every point $s$. A bad sphere is called \myemph{minimal} if it contains no other bad spheres of lower dimension as a subset. If~\eqref{eq:assump:uniqueness1} fails for a site $x$, then $x$ is equidistant from the points of a minimal bad sphere. Since the set of points that are equidistant from all points of a sphere is a proper affine subspace of $\mathbb R^d$, it suffices to show that there exist only a countable number of minimal bad spheres.

 Let $T$ be the set of atoms of $\psi$, which is countable due to locally finiteness of $\psi$. zero-dimensional bad spheres are just pairs of atoms and so they are countable. Also, a positive-dimensional minimal bad sphere $S$ contains at most one atom and so $\psi(S\backslash T)>0$. Suppose there are an uncountable number of minimal bad spheres. So, there exist $R>0$ and $\epsilon>0$ such that there are infinitely many \verified{minimal} bad spheres $\S_i, i\in\mathbb N$ such that $S_i\subseteq \ball 0R$ and $\psi(S_i\backslash T)>\epsilon$. For $i\neq j$, since $S_i\cap S_j$ is not a bad sphere, we have $\psi\left((S_i\backslash T) \cap (S_j\backslash T)\right)=0$. Therefore
$
\psi(\ball 0R)\geq \sum_i \psi(S_i\backslash T)=\infty,
$
a contradiction.
\end{proof}

\new{
\begin{proof}[Proof of Proposition~\ref{prop:spatialAverage}]
	The claim is a direct consequence of~\eqref{eq:thm:coupling-4} and Birkhoff's theorem.
\end{proof}
}

\begin{proof}[Proof of Theorem~\ref{thm:monotonicity}]

 \eqref{item:monotonicity1}
Suppose the statement is false and let $n$ be the first stage that $f(x,\xi)+R_{n}(x,\xi)> 1$ for a positive $(\mu\otimes\psi)$-measure of pairs $(x,\xi)$. By Fubini's theorem and Assumption~\ref{assump:uniqueness} for $\mu$ and $\psi$ we can find a set $\Xi_1$ with $\psi(\Xi_1)>0$ such that for each center $\xi_0\in\Xi_1$, the set
\[T:={T_{\xi_0}:=}\{x: f(x,\xi_0)+R_{n}(x,\xi_0)>1\}\]
has positive $\mu$-measure and
\begin{equation}
\label{eq:assump:uniqueness4}
\forall r>0, \exists s\in\partial \ball {\xi_0}r : \mu(\partial\ball {\xi_0}r \backslash \{s\})=0.
\end{equation}
By the definition of $T$, we have $R_{n}(x,\xi_0)>1-f(x,\xi_0)\geq 0$ for $x\in T$. So $\xi_0$ has rejected some weight from all sites in $T$ at stage $n$ and so $\xi_0$ is sated at that stage. Thus, if we let $B:={B_{\xi_0}:=}\ball {\xi_0}{r_n(\xi_0)}$, we have
\begin{equation}
\label{eq:lemma:monotonicity1}
\int_B A_n(x,\xi_0)-R_{n}(x,\xi_0)\varphi(dx)=1
\end{equation}
and moreover, $T$ is disjoint from the interior of $B$. Since $\int_{\mathbb R^d}f(x,\xi_0)\mu(dx)\leq 1$ and $\mu\geq\varphi$,~\eqref{eq:lemma:monotonicity1} gives
\begin{equation}
\label{eq:lemma:monotonicity2}
\int_B [A_n-R_{n}-f](x,\xi_0)\varphi(dx)\geq 0.
\end{equation}
\begin{lemma}
\label{lemma:monotonicity-lemma}
There is a subset $B'_{\xi_0}\subseteq B$ with positive $\varphi$-measure such that for all $x\in B'_{\xi_0}$ we have
\begin{enumerate}[(a)]
\item
\label{lemma:monotonicity-lemma1}
$[A_n-R_{n}-f](x,\xi_0)>0$,
\item
\label{lemma:monotonicity-lemma2}
\verified{$x$ is closer than some point in $T$  to $\xi_0$.
}
\end{enumerate}
\end{lemma}

 To prove the lemma, we consider three cases. Note that if $\mu$ is absolutely continuous w.r.t. the Lebesgue measure, then only the first case happens.

 \myemph{Case 1.} Suppose $\mu(T\cap\partial B)=0$ and thus $\mu(T\backslash B)>0$.
Since $\mu(T)>0$ and $f(x,\xi_0)>1-R_{n}(x,\xi_0)\geq 0$ for all $x\in T$, we have
$\int_T f(x,\xi_0)\mu(dx)>0$. Thus, by the assumption of this case,
\[\int_B f(x,\xi_0)\mu(dx) \leq 1-\int_T f(x,\xi_0)\mu(dx) <1.\]
Therefore, the inequality in~\eqref{eq:lemma:monotonicity2} is strict. Thus, the integrand is positive on a set with non-zero $\varphi$-measure, which is the desired set.

 \myemph{Case 2.} Suppose $\mu(T\cap\partial B)>0$ but $\varphi(T\cap\partial B)=0$. By~\eqref{eq:assump:uniqueness4} we get $\mu(\partial B\backslash T)=0$ and thus $\varphi(\partial B\backslash T)=0$. So $\varphi(\partial B)=0$
and we can replace $B$ by $B^{\circ}$ in~\eqref{eq:lemma:monotonicity1} and~\eqref{eq:lemma:monotonicity2}.
The rest of the argument is similar to the previous case since $B^{\circ}\cap T=\emptyset$.

 \myemph{Case 3.} Suppose $\varphi(T\cap\partial B)>0$. {Since $f+R_{n}>1$ on $T\times\{\xi_0\}$, we have}
\[
\int_{T\cap\partial B} [A_n-R_{n}-f](x,\xi_0)\varphi(dx)<0.
\]
On the other hand, \eqref{eq:assump:uniqueness4} gives $\varphi(\partial B\backslash T)=0$. Now~\eqref{eq:lemma:monotonicity2} gives
\[
\int_{B^{\circ}}[A_n-R_{n}-f](x,\xi_0)\varphi(dx)>0.
\]
So the integrand is positive on a set with positive $\varphi$-measure, which is the desired set. This completes the proof of Lemma~\ref{lemma:monotonicity-lemma}.

 For $x\in B'_{\xi_0}$, part~\eqref{lemma:monotonicity-lemma1} of Lemma~\ref{lemma:monotonicity-lemma} implies $f(x,\xi_0)<1$.
By the definition of $T$ we get $f(\cdot,\xi_0)>0$ on $T$. So $\xi_0$, $f$-desires $x$ because of part~\eqref{lemma:monotonicity-lemma2} of Lemma~\ref{lemma:monotonicity-lemma};
i.e. it desires all sites in $B'_{\xi_0}$.

 By Fubini's theorem for the set $\{(x,\xi):\xi\in\Xi_1,x\in B'_{\xi}\}$ {and Assumption~\ref{assump:uniqueness} for $\mu$ and $\psi$} we get that there is a site $x_0$ such that the set
\[\Xi_2:=\{\xi\in\Xi_1: x_0\in B'_{\xi}\}\]
has positive $\psi$-measure {and
\begin{equation}
\label{eq:assump:uniqueness3}
\forall r>0, \exists s\in\partial \ball {x_0}r : \psi(\partial\ball {x_0}r \backslash \{s\})=0.
\end{equation}
Note that our construction of $B'_{\xi_0}$ is given in terms of some inequalities. So the above set is measurable and Fubini's theorem is valid}.
Let $B_1$ be the smallest closed ball centered at $x_0$ (with possibly infinite radius) that contains $\Xi_2$.
Part~\eqref{lemma:monotonicity-lemma1} of Lemma~\ref{lemma:monotonicity-lemma} implies
\begin{eqnarray}
\label{eq:Xi2-1}
1-R_{n-1}(x_0,\xi) &\geq &[A_n-R_{n-1}](x_0,\xi) \\
\nonumber &\geq &[A_n-R_{n}](x_0,\xi)\\
\nonumber &>&f(x_0,\xi)\geq 0, \quad {\forall \xi\in \Xi_2.}
\end{eqnarray}
As a result, $x_0$ has applied to {all centers in $\Xi_2$} at stage $n$. Therefore, $A_n(x_0,\cdot)\equiv 1$ on $B_1^{\circ}$. Furthermore, by the choice of $n$ we could choose $x_0$ such that $\psi$-a.e. we have $f(x_0,\cdot)+R_{n-1}(x_0,\cdot)\leq 1$. Thus
\begin{eqnarray}
\label{eq:Xi2-2}
[A_n-R_{n-1}](x_0,\cdot)\geq f(x_0,\cdot) &\text{on} & B_1^{\circ}.
\end{eqnarray}

If $\psi(\Xi_2\cap B_1^{\circ})>0$,~\eqref{eq:Xi2-1} and~\eqref{eq:Xi2-2} give
\[
\int_{B_1^{\circ}}f(x_0,\xi)\nu(d\xi)<\int_{B_1^{\circ}}[A_n-R_{n-1}](x_0,\xi)\psi(d\xi)\leq 1,
\]
where the second inequality is due to the definition of $A_n$. Therefore, by Remark~\ref{rem:desire} we get that $x_0$, $f$-desires the centers in $\Xi_2\cap B_1^{\circ}$. {This gives an unstable pair for $f$, a} contradiction. So suppose $\psi(\Xi_2\cap B_1^{\circ})=0$. We should have $\psi(\Xi_2\cap \partial B_1)>0$. But~\eqref{eq:assump:uniqueness3} gives $\psi(\partial B_1 \backslash \Xi_2)=0$ and {as before we get}
\[
\int_{B_1}f(x_0,\xi)\nu(d\xi)<\int_{B_1}[A_n-R_{n-1}](x_0,\xi)\psi(d\xi)\leq 1.
\]
So $x_0$, $f$-desires the centers in $\Xi_2\cap\partial B_1$, a contradiction.

 \eqref{item:monotonicity2}
If $x$ fully applies to $\xi$, i.e. $A(x,\xi)=1$, then we have $f_s(x,\xi)=1-R(x,\xi)$ by the definition of $f_s$. Hence, the claim is a direct consequence of~\eqref{item:monotonicity1}.

 \eqref{item:monotonicity3}
{Let $X$ be the set of sites $x$ such that for $\psi$-a.e. $\xi$ the claim of~\eqref{item:monotonicity2} holds for $(x,\xi)$. By~\eqref{item:monotonicity2} we have $\mu(X^c)=0$. We prove that all sites in $X$ satisfy the claim of~\eqref{item:monotonicity3}.}
Suppose $x_0\in X$ and $\g {x_0}{f}t{\nu} > \g {x_0}{f_s}t{\psi} $. Therefore, $\g {x_0}{f_s}t{\psi}<1$. It follows that either $x_0$ is ${f_s}$-unexhausted or ${f_s}(x_0,\cdot)$ is positive somewhere outside $\ball{x_0}t$. Since ${f_s}$ is obtained by the site-optimal Gale-Shapley algorithm, in both cases $x_0$ has applied to all centers in $\ball {x_0}t$ with weight $1$ (for the first case use Lemma~\ref{lemma:unexhaustedApply} and for the second case note that $x_0$ has applied to some center outside the ball).
{The definition of $X$ implies that }
$f(x_0,\cdot)\leq {f_s}(x_0,\cdot)$, $\psi$-a.e. on $\ball{x_0}t$. Since $\nu\leq \psi$, we get $\g {x_0}{f}t{\nu}\leq \g {x_0}{f_s}t{\psi}$,
a contradiction.

 \eqref{item:monotonicity4}
By the right-continuity of $\h {\xi}f{\cdot}{\mu}$ and $\h {\xi}{f_s}{\cdot}{\varphi}$, it is enough to prove the claim for rational $t$. If this doesn't hold, we can find $t\in[0,\infty]$ and $S\subseteq{\myXi}$ such that $\psi(S)>0$ and $\h {\xi}{f}t{\mu} < \h {\xi}{f_s}t{\varphi}$ for all $\xi\in S$. For arbitrary $\xi_0\in S$,
since $\mu\geq\varphi$, the set of sites
\[T_{\xi_0}:=\{x\in \ball{\xi_0}t:f(x,\xi_0)<{f_s}(x,\xi_0)\}\]
has positive $\mu$-measure.
Moreover, we have $\h {\xi_0}{f}t{\mu}<1$ and $f(\cdot,\xi_0)<1$ on $T_{\xi_0}$. Therefore,
Remark~\ref{rem:desire} gives that $\xi_0$, $f$-desires all points of $T_{\xi_0}$.
We will show that $\xi_0$ can be chosen such that some point of $T_{\xi_0}$ also $f$-desires $\xi_0$ and contradiction follows.

 Assumption~\ref{assump:uniqueness} for $\mu$ and $\psi$,~\eqref{item:monotonicity2} and Fubini's theorem on $\{(x,\xi):\xi\in S, x\in T_{\xi}\}$ (which is measurable) imply that there exists a site $x_0$ such that
\begin{enumerate}[(a)]
\item
	\label{monotonicity-iv-a}
	\eqref{eq:assump:uniqueness3} holds,
\item
	\label{monotonicity-iv-b}	
	statement \eqref{item:monotonicity2} is valid for $x=x_0$ and $\psi$-a.e. $\xi$,
\item
	\label{monotonicity-iv-c}
	$\psi(C)>0$, where $C=\{\xi\in S: x_0\in T_{\xi}\}$.
\end{enumerate}
In fact, these conditions are satisfied by $\mu$-a.e. $x_0$.
Consider the smallest ball $B$ centered at $x_0$ (possibly with infinite radius) that contains $C$.
Since ${f_s}(x_0,\cdot)>f(x_0,\cdot)\geq 0$ on $C$, we get that $x_0$ has applied to the interior of $B$ with full weight. {Therefore,~\eqref{monotonicity-iv-b} implies}
\[f(x_0,\cdot)\leq {f_s}(x_0,\cdot),\quad \psi\text{-a.e. on } \interior{B}.\]

 \myemph{Case 1.} $\psi(C\cap \interior{B})>0$. By the above equation we get
\[
\int_{\interior{B}}f(x_0,\xi)\psi(d\xi) < \int_{\interior{B}} {f_s}(x_0,\xi)\psi(d\xi).
\]
Therefore $\int_{\interior B} f(x_0,\xi)\nu(d\xi)<1$. So $x_0$ $f$-desires all centers in $C\cap \interior B$ by Remark~\ref{rem:desire}, which gives an unstable pair for $f$, a contradiction.

 \myemph{Case 2.} $\psi(C\cap \interior B)=0$. Therefore $\psi(C\cap \partial B)>0$.
{By~\eqref{monotonicity-iv-a}}
we get $\psi(\partial B \backslash C)=0$. Since ${f_s}(x_0,\cdot)>f(x_0,\cdot)$ on $C$ we get as before
\[
\int_{B}f(x_0,\xi)\psi(d\xi) < \int_{B} {f_s}(x_0,\xi)\psi(d\xi).
\]
Therefore $\g {x_0}{f}s{\nu}<1$, where $s$ is the radius of $B$. So $x_0$ $f$-desires the centers in $C\cap\partial B$ by Remark~\ref{rem:desire}, a contradiction again.

 \end{proof}

\begin{proof}[Proof of Proposition~\ref{prop:uniqueness}]
Let $f$ be an arbitrary stable \verified{constrained} density for $\varphi$ and $\psi$. Corollary~\ref{cor:optimality} implies that for $\varphi$-a.e. site $x$ we have
\begin{equation}
\label{eq:uniqueness:1}
\g x{f_s}t{\psi} = \g xft{\psi} ,\quad \forall t\in[0,\infty].
\end{equation}
We can use this equation for $t<a_s(x)$ (where $a_s(x)$ is the application radius of $x$ in the site-optimal Gale-Shapley algorithm) and part~\eqref{item:monotonicity2} of Theorem~\ref{thm:monotonicity} to obtain that 
for $(\varphi\otimes\psi)$-a.e. $(x,\xi)$,
if $\abs{x-\xi}<a_s(x)$, then $f_s(x,\xi)=f(x,\xi)$.
Assumption~\ref{assump:uniqueness} and~\eqref{eq:uniqueness:1} for $t=a_s(x)$ imply that this is also valid for $\abs{x-\xi}=a_s(x)$; i.e.
for $(\varphi\otimes\psi)$-a.e. $(x,\xi)$,
if $\abs{x-\xi}\leq a_s(x)$, then $f_s(x,\xi)=f(x,\xi)$.
Since $f_s(x,\cdot)\equiv 0$ outside $\ball x{a_s(x)}$, we use~\eqref{eq:uniqueness:1} for $t=\infty$ to obtain 
{that $(\varphi\otimes\psi)$-a.e. we have $f_s=f$.
This proves the claim.}
\end{proof}

\begin{proof}[Proof of Theorem~\ref{thm:uniqueness}]
We should prove that \verified{almost surely} any two stable \verified{constrained} densities are equal except on a set with zero $(\Phi_{\omega}\otimes\Psi_{\omega})$-measure. Let $f_s$ and $f_c$ be \verified{the site-optimal and the center-optimal densities} respectively. We take expectations in Corollary~\ref{cor:optimality} and apply Lemma~\ref{lemma:doubleCounting} to get
\begin{eqnarray*}
\lambda_{\Phi}\omidPalm{\Phi}{\g 0{f_s}t{\Psi}}& \geq & \lambda_{\Phi}\omidPalm{\Phi}{\g 0{f_c}t{\Psi}} = \lambda_{\Psi}\omidPalm{\Psi}{\h 0{f_c}t{\Phi}}\\
&\geq & \lambda_{\Psi}\omidPalm{\Psi}{\h 0{f_s}t{\Phi}}=\lambda_{\Phi}\omidPalm{\Phi}{\g 0{f_s}t{\Psi}}.
\end{eqnarray*}
Therefore, all inequalities are indeed equality. Hence \[\omidPalm{\Phi}{\g 0{f_s}t{\Psi} - \g 0{f_c}t{\Psi}} = 0.\] By an argument similar to~\eqref{eq:thm:coupling-2} we get
\[
\omid{\int_{\mathbb R^d}\g x{f_s}t{\Psi}-\g x{f_c}t{\Psi}\Phi(dx)}=0.
\]

 Since the integrand is non-negative by Corollary~\ref{cor:optimality}, for a given $t$, we almost surely have for $\Phi_{\omega}$-almost every $x\in\mathbb R^d$
\[
\g x{f_s}t{\Psi} = \g x{f_c}t{\Psi}.
\]
Considering this for rational $t$ and using right-continuity of $\g x{f_s}{\cdot}{\Psi}$ and $\g x{f_c}{\cdot}{\Psi}$ we get that almost surely for $\Phi_{\omega}$-a.e. $x$ we have
\begin{equation*}
\g x{f_s}t{\Psi} = \g x{f_c}t{\Psi},\quad \forall t\in[0,\infty].
\end{equation*}

 Now, for a sample $\omega\in\Omega$ such that the above equation holds for $\Phi_{\omega}$-a.e. $x$ and Assumption~\ref{assump:uniqueness} holds for $\Phi_{\omega}$ and $\Psi_{\omega}$, the claim is a direct consequence of Proposition~\ref{prop:uniqueness}.
\end{proof}

	\begin{proof}[Proof of Lemma \ref{lemma:star}]
	\verified{
		\verified{Remark~\ref{rem:VoronoiIff} easily implies that the Voronoi territory of $\xi$ is star-shaped with center $\xi$.}
		As an example {for non-convex territories,} let $\psi$ be half of the counting measure on the vertices of an equilateral triangle in the plane. \verified{Also, if $\psi=\mathcal L+\delta_0$, we see that the Voronoi territory of center $0$ is not closed.} 
		
		
		
		Now, suppose the support of $\psi$ is a discrete set $\{\xi_1,\xi_2,\ldots\}$. 
		\verified{
			Let $D_i$ be the Voronoi territory of a center $\xi_i$.
			In this case, Remark~\ref{rem:VoronoiIff} gives that $x\in D_i$ if and only if $\psi(\ballint{x}{\dist x{\xi_i}})<1$. This easily implies that $D_i$ is closed.
		}
		With the notations of Definition~\ref{def:vtk}, let $A$ be the set of points $x\in\mathbb R^d$ such that $\partial \ball{x}{s(x)}$ contains more than one atom of $\psi$. In each component of $A^c$, the set of atoms in $\ball{x}{s(x)}$ is fixed. Therefore, all Voronoi territories are bound by $A$. Moreover, $A\subseteq \cap_{i,j}A_{i,j}$, where $A_{i,j}:=\{x:\{\xi_i,\xi_j\}\subseteq\partial\ball{x}{s(x)}\}$. Each $A_{i,j}$ is contained in a hyperplane. Therefore, it suffices to prove that for each compact set $K$, only finitely many of the $A_{i,j}$'s hit $K$. Let $S\supset K$ be a compact set such that $\psi(S)>1$ (in the case $\psi(\mathbb R^d)=1$ the claim is trivial). Each open ball $\ballint{x}{s(x)}$ for $x\in K$ doesn't contain $S$ and therefore it is contained in a compact set $T$ which depends only on $K$ and $S$. Now, if $A_{i,j}$ hits $K$, then $\xi_i$ and $\xi_j$ lie in $T$. It follows that such pairs $(i,j)$ are finite and we are done.
	}
	\end{proof}

\begin{proof}[Proof of Proposition~\ref{prop:vterritories}]
%

Suppose the Voronoi territory $C$ of center $\xi_0$ is unbounded.
According to Lemma~\ref{lemma:star}, $C$ is star-shaped and hence there is a half-line $l$ starting at $\xi_0$ which completely lies in $C$. For $x\in l$, \verified{Remark~\ref{rem:VoronoiIff} gives} $\psi(\ballint x{\dist x{\xi_0}})\leq 1$. By taking union over $x\in l$, we find $\Psi_{\omega}(H)\leq 1$, where $H$ is an open half-space orthogonal to $l$, a contradiction. 

\verified{Now, suppose $\Psi$ is a stationary non-zero random measure. One can obtain from Lemma~\ref{lemma:infinite} that 
	every half-space has infinite $\Psi$-measure a.s. and the claim follows.}
%
\end{proof}

\begin{lemma}
\label{lemma:JohnsonMehl}
Let $A\subseteq\mathbb R^d$, $a \in \mathbb R^d$ and $r>0$. Let $C$ be the set of points $x$ such that $\dist xa < \dist x{a'}+r$ for all $a'\in A$. $C$ is bounded if and only if $\ball ar$ is contained in the interior of $\conv{A}$, where $\conv{A}$ stands for the convex hull of $A$.
\end{lemma}

 \begin{proof}
One has $C$ is star-shaped with center $a$. Hence, it is unbounded if and only if it contains a half-line $\{a+tv:t\in[0,\infty)\}$ for some unit vector $v\in\mathbb R^d$. {This is equivalent to $A\subseteq \{x:v\cdot(x-a)\leq r\}$ for some unit vector $v$. Equivalently, $\ball a r \not\subseteq\interior{\conv A}$.
}
\end{proof}
\begin{proof}[Proof of Theorem~\ref{thm:bounded}]
Using the ergodic decomposition theorem,
we can assume that $\mathbb P$ is an ergodic measure {for the family $\left(\theta_s\right)_{s\in\mathbb R^d}$}.

 \eqref{thm:bounded:1} We call a center \myemph{bad} if its territory is unbounded. Let $\Xi'$ be the set of bad centers.
Suppose $\Psi(\Xi')\neq 0$ with positive probability. By ergodicity, this happens almost surely. Since $F$ is flow-adapted, the restriction $\Psi'$ of $\Psi$ to $\Xi'$ is an ergodic random measure with positive intensity.
Let $\xi_0\in\Xi'$ and let $C$ be the \vterritory{} of $\xi_0$ with respect to \verified{$\frac 12 \Psi'$}. We claim that the territory of $\xi_0$ is a subset of $C$ and contradiction follows directly by applying Proposition~\ref{prop:vterritories} for $\Psi'$. If the claim is not true, let $x_0\not\in C$ such that $F(x_0,\xi_0)>0$. \verified{By Remark~\ref{rem:VoronoiIff}  
we get $\Psi'(\ballint{x_0}{\dist{x_0}{\xi_0}})\geq 2>1$.} Therefore, there is a bad center $\xi_1\in \ballint{x_0}{\dist{x_0}{\xi_0}}$ such that $F(x_0,\xi_1)<1$. Now $x_0$ desires $\xi_1$ since it prefers $\xi_1$ over $\xi_0$. Also $\xi_1$ desires $x_0$ since its territory is unbounded. So $(x_0,\xi_1)$ is an unstable pair, a contradiction.

\eqref{thm:bounded:2} By stationarity, it is enough to show that for a deterministic point $z$, the union of the territories of the sites in $\ball z1$ has finite $\Psi$-measure almost surely. Call a lattice point $z\in \mathbb Z^d$ \myemph{bad} if this property doesn't hold for $z$. Suppose by contradiction that 0 is bad with positive probability. Then the set of bad lattice points form an ergodic \verified{simple} point process on $\mathbb Z^d$ with positive intensity. 
Let $Z$ be the set of bad lattice points, which contains $z_0:=0$.
Since every half-space contains some element of $Z$, we have $Conv(Z)=\mathbb R^d$. Therefore, we can find $z_1,\ldots,z_n\in Z\backslash\{0\}$ such that $\ball {z_0}2\subseteq\conv{z_1,\ldots,z_n}$. Let $C_0$ be the region defined in Lemma~\ref{lemma:JohnsonMehl} for $A=\{z_1,\ldots,z_n\}$, $a=z_0$ and $r=2$, which is bounded by the claim of the lemma. For $i=1,\ldots,n$ let $C_i$ be the set of points $x\not\in C_0$ that are closer to $z_i$ than other points of $A$.
Let $T$ be the union of the territories of the sites in $\ball{z_0}1$. Since $z_0$ is a bad lattice point, we have $\Psi(T)=\infty$. Since $\Psi(C_0)<\infty$, there is $i>0$ and a bounded Borel set $D\subseteq T\cap C_i$ such that $\Psi(D)>1$. Since $z_i$ is a bad lattice point, there is a site $x_i\in \ball {z_i}1$ such that its territory contains some centers further away than all of the centers in $D$. Since $\Psi(D)>1$, we can find a center $\xi_i\in D$ such that $F(x_i,\xi_i)<1$. We claim that $(x_i,\xi_i)$ is an unstable pair.

 Since $\xi_i\in T$, there is a site $x_0\in\ball {z_0}1$ such that $F(x_0,\xi_i)>0$. The fact that $\xi_i\not\in C_0$ gives $\dist{\xi_i}{z_0} \geq \dist{\xi_i}{z_i}+2$. Therefore $\dist{\xi_i}{x_0}>\dist{\xi_i}{x_i}$ and thus $\xi_i$ desires $x_i$. On the other hand, $x_i$ desires $\xi_i$ since its territory contains centers further away than $\xi_i$ by the definition of $x_i$. So $(x_i,\xi_i)$ is an unstable pair, a contradiction.
\end{proof}

 \section{Examples}
\label{sec:examples}
\begin{figure}[t]
	\centering
\subfigure[$A_1$]{
		\label{fig:interval1}
\includegraphics[width=.3\textwidth]{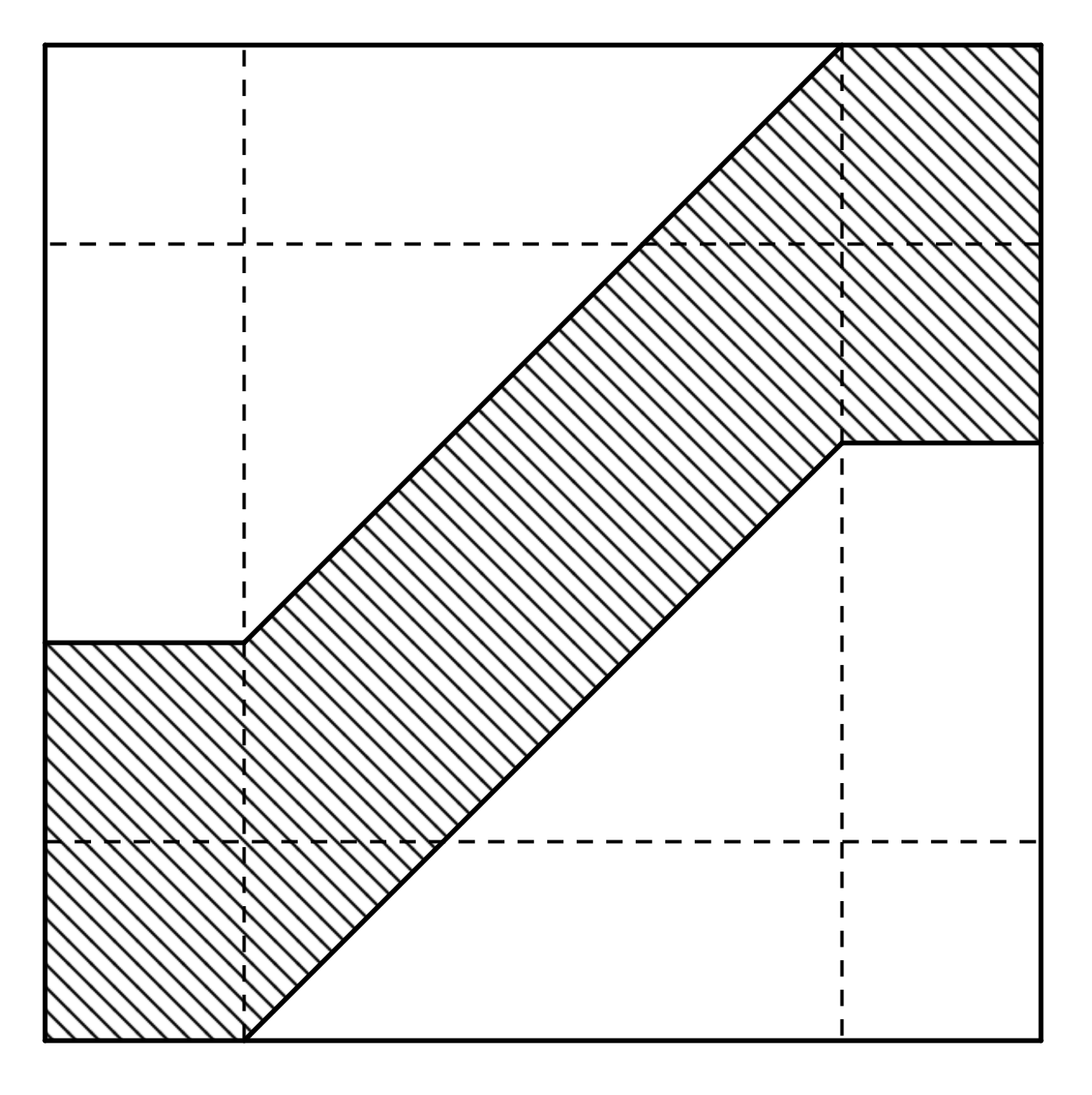}
}
\subfigure[$A_1-R_1$]{
		\label{fig:interval2}
\includegraphics[width=.3\textwidth]{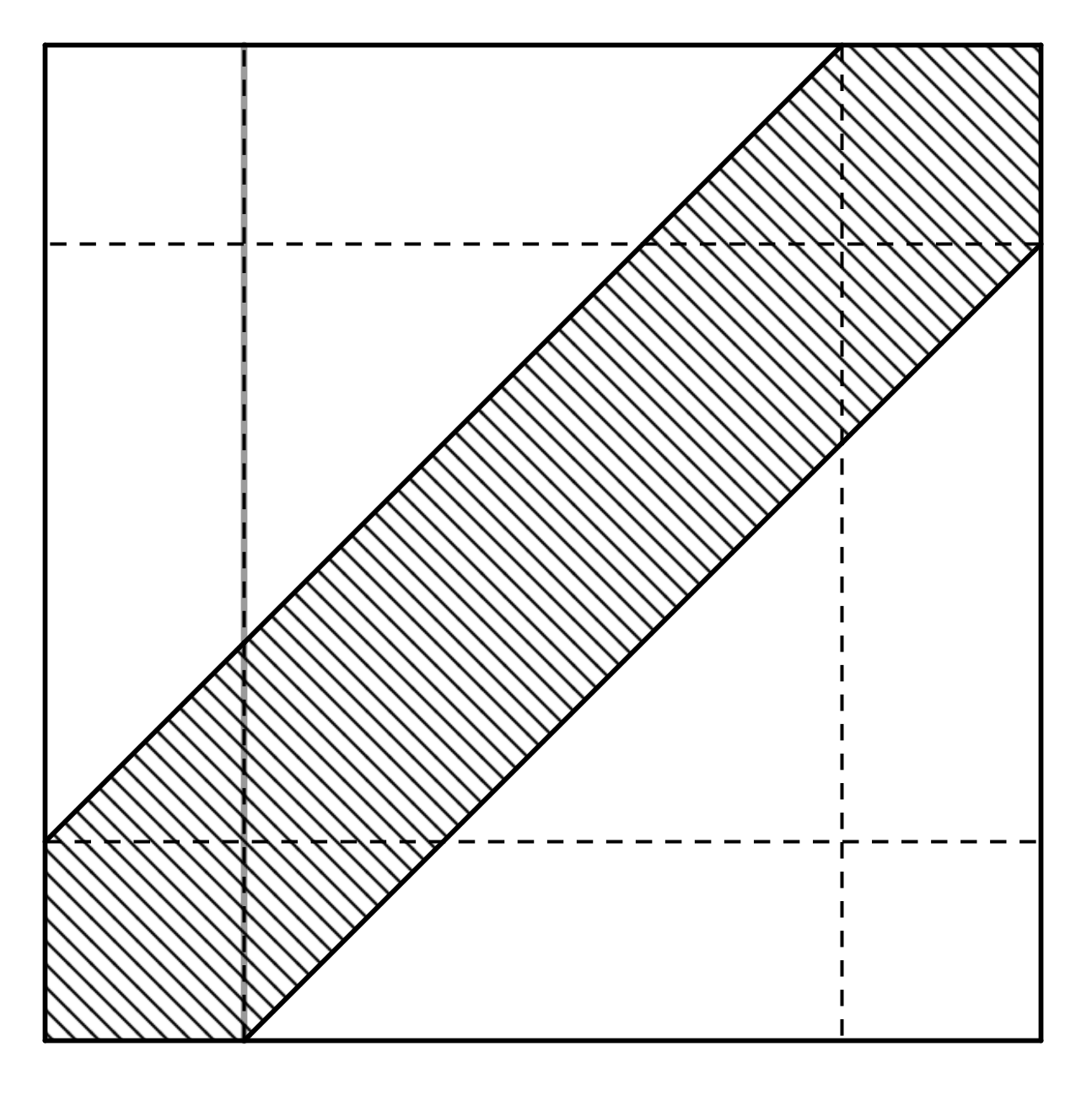}
}
\subfigure[$f$]{
		\label{fig:interval3}
\includegraphics[width=.3\textwidth]{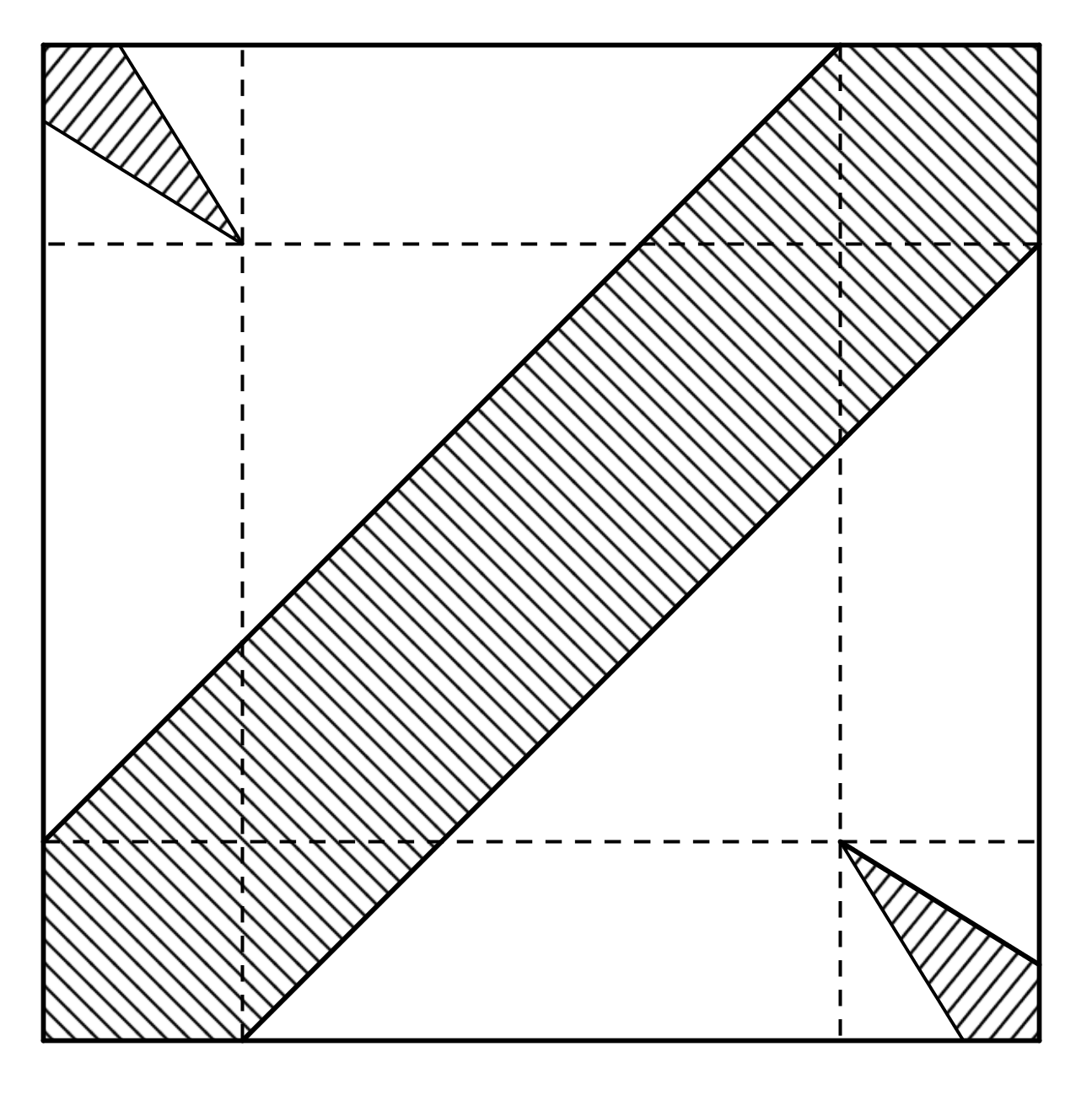}
}
\caption{The supports of three functions in the setting of Example~\ref{ex:interval}, which map $[0,\alpha]\times[0,\alpha]$ to $\{0,1\}$. The first axis stands for sites and the second axis stands for centers. The dashed lines are the lines with distance $\frac 1{2}$ to the sides.
The small segments in the third figure have slope $-\phi$ or $-\frac 1{\phi}$, where $\phi=\frac{1+\sqrt 5}2$.}
\label{fig:interval}
\end{figure}


In the following examples we let $\mathcal Z_d$ be the counting measure on $\mathbb Z^d$; i.e. \[\mathcal Z_d = \sum_{z\in \mathbb Z^d}{\delta_z}.\]
Although $\mathcal Z_d$ is not a stationary measure on $\mathbb R^d$, in examples~\ref{ex:Z*R} and~\ref{ex:Z+R} we can translate the measures by a random uniform element of $[0,1]^d$ to obtain ergodic {stationary} random measures. Therefore, {the claims of Theorem~\ref{thm:coupling}, Theorem~\ref{thm:site-optimalBalancing} and Theorem~\ref{thm:uniqueness} are valid} in these examples.

\begin{example}
\label{ex:generalization}
When $\varphi=\mathcal L_d$ and $\psi$ is a counting measure, \verified{the transport kernel given by the site-optimal density via Remark~\ref{rem:kernel}} coincides a.e. with the allocation presented in~\cite{HoPe06}.
In this setting, definitions~\ref{def:compatibleDensity} and~\ref{def:stable} also generalize the definitions of allocations and stability in~\cite{HoPe06}. 
\end{example}

 \begin{example}
\label{ex:interval}
Let $\varphi$ and $\psi$ be both the Lebesgue measure restricted on $[0,\alpha]$, where $\alpha\geq \frac 32$. Figure~\ref{fig:interval1} and~\ref{fig:interval2} illustrate the first stage of Algorithm~\ref{alg:Gale}. It can be seen that the site-optimal and the center-optimal densities agree a.e. with the function depicted in Figure~\ref{fig:interval3}. Therefore, we can use Proposition~\ref{prop:uniqueness} to see that this is the unique stable \verified{constrained} density for $\varphi$ and $\psi$. Note that the sites (resp. centers) that have distance less than $\frac 1{2}(1-\frac 1{\phi})$ to the boundary are unexhausted (resp. unsated), where $\phi$ is the golden number; i.e. it satisfies $\phi-\frac 1{\phi}=1$.
\end{example}

 \begin{example}
\label{ex:ZR}
Let $\varphi=\mathcal L_1$ and \verified{$\psi$ be the counting measure on a subset $A\subseteq \mathbb Z$. For $x\in\mathbb R$ and $\xi\in A$} one has
\begin{equation*} 
f_s(x,\xi)=\left\{
\begin{array}{ll}
1, & \verified{\dist x {\xi}}\leq\frac 12 \\
0, & \text{ otherwise }
\end{array}
\right.
\end{equation*}
\verified{It is easy to verify that $f_s$ is a balancing stable \verified{constrained} density only when $A=\mathbb Z$.} 
\end{example}

 \begin{example}
\label{ex:Z*R}
Let $\varphi=\mathcal L_2$ and $\psi=\mathcal Z_1 \otimes \mathcal L_1$.
Define
\[
f(x,\xi)=\left\{
\begin{array}{ll}
1, & \max \{\abs{x_1-\xi_1}, \abs{x_2-\xi_2}\}\leq\frac 12,\\
0, & \text{otherwise.}
\end{array}
\right.
\]
This function is a balancing \verified{constrained} density, but is not stable since the center $\xi_0=(0,0)$ and the site $x_0=(a,0)$ desire each other for $\frac 12 < a < \frac 58$. 
However, Algorithm~\ref{alg:Gale} gives $f_s$ in one step, which is a balancing stable \verified{constrained} density. For $x=(x_1,x_2)$ and $\xi=(\xi_1,\xi_2)$, the site-optimal density is
\begin{equation*}
f_s(x,\xi)=\left\{
\begin{array}{ll}
1, & \abs{x_2-\xi_2} \leq \min\{\frac 12, \frac 54 - 2\abs{x_1-\xi_1}\},\\
0, & \text{ otherwise.}
\end{array}\right.
\end{equation*}

 \begin{figure}[t]
\centering
\includegraphics[width=.9\textwidth]{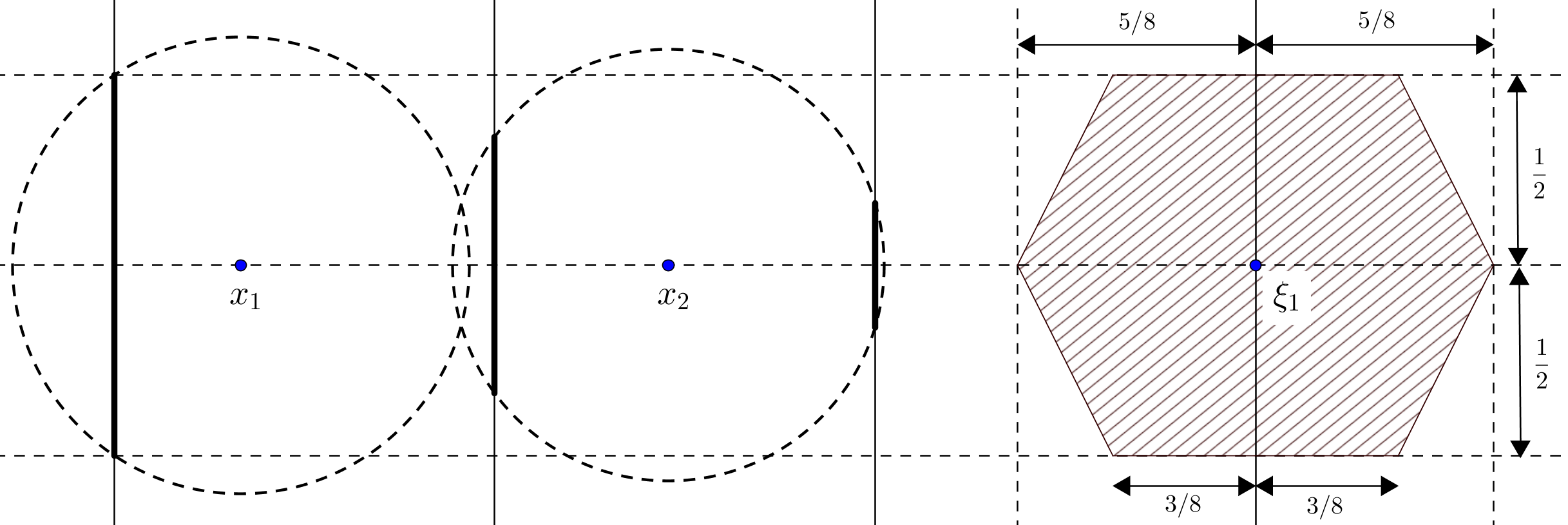}
	\caption{The territories of two sites $x_1$ and $x_2$ and a center $\xi_1$ in Example~\ref{ex:Z*R}. The vertical solid lines represent $\mathbb Z \times \mathbb R$.
	The dashed circles show the application {radii} of the two sites and the vertical bold lines show their territories.}
	\label{fig:ex:Z*R}
\end{figure}
The territory of each center is a hexagon 
as illustrated in Figure~\ref{fig:ex:Z*R}.
%
\end{example}

 \begin{example}
\label{ex:Z+R}
Let $\varphi=2\mathcal L_1$ and $\psi=\mathcal L_1 + \mathcal Z_1$. For $0\leq x\leq\frac 12$, the site-optimal density is
\begin{equation*}
f_s(x,\xi)=\left\{
\begin{array}{ll}
1, & 0<\xi\leq 2x,\\
1-2x, & \xi=0,\\
0, & \text{ otherwise}
\end{array}\right.
\end{equation*}
and for $-\frac 12\leq x\leq 0$ we have
\begin{equation*}
f_s(x,\xi)=\left\{
\begin{array}{ll}
1, & 2x\leq\xi<0,\\
1-2\norm x, & \xi=0,\\
0, & \text{ otherwise.}
\end{array}\right.
\end{equation*}
A similar equation holds for other values of $x$. By applying a random translation, Theorem~\ref{thm:site-optimalBalancing} holds and $f_s$ is a balancing stable \verified{constrained} density.
\end{example}

\begin{example}
\label{ex:sqrt2}
\verified{Let $\Phi$ and $\Psi$ be jointly stationary ergodic counting measures (i.e. simple point processes) in $\mathbb R$, with positive and finite intensities $\lambda_{\Phi}$ and $\lambda_{\Psi}$. The random measures $\Phi':=\frac 1{\lambda_{\Phi}}\Phi$ and $\Psi':=\frac 1{\lambda_{\Psi}}\Psi$ have unit intensity, but there is no $(\Phi',\Psi')$-balancing allocation provided that $\frac{\lambda_{\phi}}{\lambda_{\psi}}\not\in\mathbb Z$ (note that $\frac{\lambda_{\phi}}{\lambda_{\psi}}$ should be the number of pre-images of a center). However, Theorem~\ref{thm:site-optimalBalancing} shows that there is a flow-adapted balancing transport kernel between them.
}
\end{example}

\verified{
\begin{example}
\label{ex:noAllocation}
In the setting of Example~\ref{ex:sqrt2}, let $\Phi'':=\Phi'\times\mathcal L_{d-1}$ and $\Psi'':=\Psi'\times\mathcal L_{d-1}$, which are stationary ergodic random measures on $\mathbb R^d$ with unit intensity. Theorem~\ref{thm:site-optimalBalancing} shows the existence of a flow-adapted $(\Phi'',\Psi'')$-balancing transport kernel, but we claim that there is no such allocation. 

Let $\tau$ be a flow-adapted $(\Phi'',\Psi'')$-balancing allocation. By translation invariance, we get that for a site $x$, the vector $\tau_{\omega}(x)-x$ depends only on $\pi_1(x)$ and $\omega$, where $\pi_1$ is projection on the first coordinate. It follows easily that the allocation on $\mathbb R$ defined by $x\mapsto \pi_1(\tau((x,0)))$ is $(\Phi',\Psi')$-balancing, which is a contradiction by Example~\ref{ex:sqrt2}.
\end{example}
}

\begin{example}
\label{ex:-+}
Let $\varphi$ be the Lebesgue measure on $(0,\infty)$ and $\psi$ be the Lebesgue measure on $(-\infty,0)$. It is easy to see that $f_s(x,\xi)\in\{0,1\}$ and $f_s(x,\xi)=1$ if and only if $\lceil x \rceil = \lceil{-\xi}\rceil$, where $\lceil{a}\rceil$ is the smallest integer not less than $a$.
\end{example}

 \ACKNO{\verified{We are grateful to Herman Thorisson for his helpful comments on the shift-coupling problem.} \verified{We also thank the anonymous reviewers for their valuable comments on this work.}}

\end{document}